\documentclass[reqno]{amsart}
\usepackage{tkz-euclide}
\usepackage{hyperref}
\usepackage{amsmath}
\usepackage{amssymb} 
\usepackage{mathrsfs}
\usepackage{graphicx}
\usepackage{tikz}
\usepackage[font=footnotesize, labelfont=bf]{caption}
\usepackage{subcaption}
\usepackage{xcolor}  
\usepackage{amsthm}
\usepackage{float}
\usepackage{enumitem}
\usepackage[english]{babel}

\definecolor{LimeGreen}{cmyk}{0.50, 0.5, 1, 0}

\newcommand{\eps}{\varepsilon} 
\newcommand{\dx}{\, {\rm d}x}
\newcommand{\dy}{\, {\rm d}y}
\newcommand{\dz}{\, {\rm d}z}
\newcommand{\ds}{\, {\rm d}s}
\newcommand{\dt}{\, {\rm d}t}

\newcommand{\e}{\varepsilon}

\DeclareMathOperator{\dist}{dist}
\newcommand{\Sph}{{\mathbb S}}

\newcommand{\eg}{{\it e.g.}, }
\newcommand{\ie}{{\it; i.e.}, }
\newcommand{\cf}{{\it cf.\ }}

\theoremstyle{plain}
\begingroup
\newtheorem{theorem}{Theorem}[section]
\newtheorem{lemma}[theorem]{Lemma}
\newtheorem{proposition}[theorem]{Proposition}
\newtheorem{corollary}[theorem]{Corollary}
\endgroup

\numberwithin{equation}{section}

\newcommand{\N}{\mathbb{N}}
\newcommand{\Z}{\mathbb{Z}}
\newcommand{\Q}{\mathbb{Q}}
\newcommand{\R}{\mathbb{R}}

\newcommand{\F}{\mathscr{F}}
\newcommand{\E}{\mathscr{E}}

\newcommand{\Adm}{\mathscr{A}}
\newcommand{\A}{\mathcal{A}}
\renewcommand{\S}{\mathbb{S}}
\renewcommand{\L}{\mathcal{L}}
\renewcommand{\H}{\mathcal{H}}

\newcommand{\dHn}{\, {\rm d}\H^{n-1}}

\newcommand{\loc}{\mathrm{loc}}

\newcommand{\x}{\times }
\newcommand{\m}{\mathbf{m}}

\renewcommand{\hom}{\mathrm{hom}}
\newcommand{\defas}{:=}
\newcommand{\wto}{\rightharpoonup}

\newcommand{\wcont}{\subset\subset}
\newcommand{\LtL}{L^0(\R^n;\R^m)\times L^0(\R^n)}

\newcommand{\uu}{{\rm u}}
\newcommand{\vv}{{\rm v}}

\newenvironment{step}[1]{\medskip\textit{Step #1:}}{}

\theoremstyle{definition}
\begingroup

\endgroup
\theoremstyle{remark}
\newtheorem{remark}[theorem]{Remark}
\renewcommand{\tilde}{\widetilde}
\newcommand{\sm}{\setminus}

\renewcommand{\d}{\, \mathrm{d}}
\usepackage[normalem]{ulem}

\title[Gradient damage models for heterogeneous materials]
{Gradient damage models for heterogeneous materials}

\author[A. Bach]{Annika Bach}
\address[A. Bach]{Dipartimento di Matematica ``Guido Castelnuovo'', Sapienza Universit\`{a} di Roma, Italy}
\email{annika.bach@uniroma1.it}
\author[T. Esposito]{Teresa Esposito}
\address[T. Esposito]{}
\email{teres.esposito@gmail.com}
\author[R. Marziani]{Roberta Marziani}
\address[R. Marziani]{Angewandte Mathematik, WWU M\"unster, Germany}
\email[Roberta Marziani]{roberta.marziani@uni-muenster.de}
\author[C. I. Zeppieri]{Caterina Ida Zeppieri}
\address[C. I. Zeppieri]{Angewandte Mathematik, WWU M\"unster, Germany}
\email[Caterina Zeppieri]{caterina.zeppieri@uni-muenster.de}

\begin{document}
	

	\begin{abstract}
	In this paper we study the asymptotic behaviour of phase-field functionals of Ambrosio and Tortorelli type allowing for small-scale oscillations both in the volume and in the diffuse surface term. The functionals under examination can be interpreted as an instance of a static gradient damage model for heterogeneous materials. Depending on the mutual vanishing rate of the approximation and of the oscillation parameters, the effective behaviour of the model is fully characterised by means of $\Gamma$-convergence.     
		
	\end{abstract}
	
	\maketitle
	
	{\small
	\noindent \keywords{\textbf{Keywords:} gradient damage model, phase-field functionals, homogenisation, $\Gamma$-convergence, elliptic approximation, free-discontinuity functionals.}
	
	\medskip
	
	\noindent \subjclass{\textbf{MSC 2010:} 
49J45, 
49Q20,  
74Q05.  
}}

	\section{Introduction}
	Damage models for elastic materials describe the degradation of the elastic properties of a body as a consequence of some applied loads \cite{PAMM, PM2}. The total energy of an elastic material undergoing damage then depends on two variables: the deformation $u:A\to\R^m$ ($A\subset\R^n$ open and bounded representing the reference configuration of the material) and an internal variable $v:A\to[0,1]$ measuring at each point the damage state of the material (the value $v=1$ corresponding to the original sound state and the value $v=0$ corresponding to the totally damaged state).
%
%
%
In a periodically heterogeneous setting and at fixed time, this energy can be described in terms of phase-field functionals of the form 	
\begin{equation}\label{intro:en}
\F_\e(u,v,A) = \int_A(v^2+\eta_\e)f\Big(\frac{x}{\delta_\e},\nabla u\Big)\dx + \frac{1}{\e} \int_A (1-v)^2 \dx + \e \int_A h\Big(\frac{x}{\delta_\e},\nabla v\Big)\dx\,,	
\end{equation}
where $\e>0$ is a small parameter, $0<\eta_\e \ll \e$, and $\delta_\e>0$ is infinitesimal, as $\e \to 0$. 
The integrands $f$ and $h$ are $(0,1)^n$-periodic in the first variable and satisfy growth and coercivity conditions of order $2$ in the second variable (see Section \ref{sect:setting} for the full list of assumptions). Correspondingly, $\F_\e$ is defined for $(u,v)\in W^{1,2}(A;\R^m)\times W^{1,2}(A;[0,1])$.

The three terms in~\eqref{intro:en} can be interpreted as follows. The first term accounts for the stored elastic energy and reflects the worsening of the elastic properties of the material due to the damage process. Namely, in the regions where the damage occurs, that is, where $v\simeq 0$, the deformation gradient $\nabla u$ becomes very large in norm and hence $u$ singular. 
The second term represents the energy dissipated in the damage process, hence it is maximal when the material is totally damaged. 
Together with the third term, which penalises the spatial variations of $v$, it forces the damage to localise for small $\e$ in diffuse regions of size proportional to $\e$,
around the set where $|\nabla u|$ blows up. Then, asymptotically, the damage-localisation gives rise to sharp cracks and the functionals in~\eqref{intro:en} are expected to behave, in the limit, as a fracture model. 

In a homogeneous setting, choosing $f(x,\nabla u)= |\nabla u|^2$, and $h(x,\nabla v)=|\nabla v|^2$, the functionals $\F_\e$ reduce to the classical Ambrosio-Tortorelli model which is indeed known to $\Gamma$-converge to the prototypical brittle fracture model given by the Mumford-Shah functional \cite{AT90,AT92}.
If now instead $\eta_\e \sim \e$, then the static damage model described by $\F_\e$ can be shown to $\Gamma$-converge to a fracture model of cohesive type \cite{DMI} (see also \cite{Fo-Iu, I}).

Moreover, in a heterogeneous scale-free setting \cite{Foc01}, that is, when functionals of the form 
\begin{equation}\label{intro:Foc}
\int_A v^2 f(x,\nabla u)\dx + \frac{1}{\e} \int_A (1-v)^2 \dx + \e \int_A h(x,\nabla v)\dx   
\end{equation}
are considered, the corresponding $\Gamma$-limit is given by the nonhomogeneous, anisotropic brittle energy  \`a la Griffith  
\[
\int_A f(x,\nabla u)\dx + 2 \int_{S_u \cap A} \sqrt{h(x,\nu_u)}\d\mathcal H^{n-1},   
\]
where now $u$ belongs to $GSBV^2(A;\R^m)$\ie is a generalised special function of bounded variation \cite{AFP, DMFT}. In this functional framework $\nabla u$ denotes the approximate differential of $u$, $S_u$ its discontinuity set, and $\nu_u$ the normal to $S_u$.  

In the recent work \cite{BMZ} the authors analysed the limit behaviour of a general class of heterogeneous, scale-dependent phase-field functionals of the form  
\begin{equation}\label{intro:BMZ}
\int_A v^2 f_\e(x,\nabla u)\dx + \frac{1}{\e} \int_A g_\e(x,v,\e\nabla v)\dx   
\end{equation}
where $f_\e$ and $g_\e$ belong to suitable classes of integrands including, in particular, the choices 
\[
f_\e(x,\nabla u)=f\Big(\frac{x}{\delta_\e},\nabla u\Big), \quad g_\e(x,v,\nabla v)= (1-v)^2 + h\Big(\frac{x}{\delta_\e},\nabla v\Big),
\]
with $f$ and $h$ as in \eqref{intro:en}. The main result in \cite{BMZ} establishes that the functionals in \eqref{intro:BMZ} essentially behave like \eqref{intro:Foc}. That is, the (possibly sequence-dependent) $\Gamma$-limit of \eqref{intro:BMZ} is a free-discontinuity functional of the form 
\[
\int_A f_\infty(x,\nabla u)\dx + \int_{S_u \cap A} g_\infty(x,\nu_u)\d\mathcal H^{n-1},   
\]  
where $f_\infty$ and $g_\infty$ can be characterised in terms of limits of suitable scaled minimisation problems. Furthermore, the minimisation problem   
providing $f_\infty$ only involves $f_\e$ while the formula providing $g_\infty$ only involves the minimisation of $g_\e$ (over pairs $(u,v)$ along which the first term in \eqref{intro:BMZ} vanishes).   In this respect, the $\Gamma$-convergence result in \cite{BMZ} can be seen as the phase-field analogue of the decoupling result for free-discontinuity 
functionals proven in \cite{CDMSZ19}. In particular, we notice that in the regular elliptic setting the limit decoupling immediately implies that the $\Gamma$-limit of  \eqref{intro:BMZ} does not depend on the jump opening of $u$ and hence is a brittle energy.

Building upon the limit decoupling obtained in \cite{BMZ} it can be proven (\emph{cf.}\ Theorem \ref{thm:gamma_limit} and Proposition \ref{prop:volume_density}) that the functionals $\F_\e$ $\Gamma$-converge to a free-discontinuity functional of the form         
\begin{equation}\label{intro:G-lim}
\int_A f_\hom(\nabla u)\dx + \int_{S_u \cap A} g^\ell_\hom(\nu_u)\d\mathcal H^{n-1}, 
\end{equation}
where $f_\hom$ is given by the classical formula of periodic homogenisation \cite{Braides85, Mueller87}\ie
\[
f_{\rm hom}(\xi)= \lim_{r\to+\infty}\frac{1}{r^n}\inf\Biggl\{\int_{Q_r(0)}f(x,\nabla u)\dx\colon u\in W^{1,2}(Q_r(0);\R^m)\,,u=u_\xi\ \text{near}\ \partial Q_r(0) \Biggr\}, 
\]
while the homogeneous surface energy density $g^\ell_\hom$ will depend on the parameter 
\[
\ell=\lim_{\e \to 0}\frac{\e}{\delta_\e}\in [0,+\infty],
\] 
or, in other words, on the mutual vanishing rate of the approximation and the oscillation parameters. Then, the main contribution of the present paper consists in the characterisation of $g^\ell_\hom$ in the three limit regimes: $\ell=0$, $\ell \in (0,+\infty)$, and $\ell=+\infty$ ({\it cf.}\ Theorem \ref{thm:main_theorem}, Propositions \ref{prop:supercritical_case}, \ref{prop:critical_case}, and \ref{prop:subcritical:case}). 

We notice that in the one-dimensional setting\ie for $n=m=1$, the computation of the $\Gamma$-limit of the damage model \eqref{intro:en} can be carried out directly, by hands, without resorting to the general convergence result in \cite{BMZ}, as shown in \cite{BEMZ}.  
On the other hand, we observe that the analysis to determine $g^\ell_\hom$ shares some similarities with the homogenisation of phase-transition functionals of Modica-Mortola type as in \cite{ABC01}. We also mention here the work \cite{Morfe} where the stochastic analogue of \cite{ABC01} is considered, though only for $\delta_\e=\e$. Moreover, in the papers \cite{BZ,CFG,CFHP,M} further variants of phase-transition functionals are considered where a transition-scale and an oscillation-scale appear at the same time and interact in the $\Gamma$-limit.    

We now give a brief heuristic account of the $\Gamma$-convergence result, Theorem \ref{thm:main_theorem}, in the three regimes  $\ell=0$, $\ell \in (0,+\infty)$, and $\ell=+\infty$. As already observed, thanks to the general analysis developed in \cite{BMZ}, it is not difficult to show that the functionals $\F_\e$ behave like 
\begin{equation}\label{intro:eff-funct}
 \int_A(v^2+\eta_\e)f_\hom(\nabla u)\dx + \frac{1}{\e} \int_A (1-v)^2 \dx + \e \int_A h\Big(\frac{x}{\delta_\e},\nabla v\Big)\dx.
 \end{equation}
 Then, if $\ell=0$, which corresponds to the case $\e \ll \delta_\e$, we can regard $\delta_\e$ and the variable $x/\delta_\e$ as fixed and let first $\e \to 0$. In this way, arguing as in the case of \eqref{intro:Foc}
 we would get the $\delta_\e$-dependent free-discontinuity functionals 
 \begin{equation}\label{intro:sub}
 \int_A f_\hom(\nabla u)\dx+ 2 \int_{S_u \cap A} \sqrt{h\Big(\frac{x}{\delta_\e},\nu_u\Big)}\d\mathcal H^{n-1}.   
 \end{equation} 
 Therefore, appealing to \cite{BDV}, letting $\delta_\e \to 0$ yields the homogeneous free-discontinuity function \eqref{intro:G-lim} with surface energy density $g_\hom^0$ given by
 \[
\begin{split}
g^0_{\rm hom}(\nu)= 2 \lim_{r\to+\infty}\frac{1}{r^{n-1}}\inf\Biggl\{\int_{S_u\cap Q^\nu_r(0)}&\sqrt{h(x,\nu_u)}\dHn\colon \\
&u\in BV(Q_r^\nu(0);\{0,e_1\})\,, u=u^\nu\ \text{near}\ \partial Q^\nu_r(0)\Biggr\}\,,
\end{split}
 \]
 where $u^\nu$ denotes the jump function defined as
 \begin{equation*}
			u^{\nu}(x)=\begin{cases} e_1 & \text{if }\; x\cdot \nu \geq 0\,, 
				\cr
				0 & \text{if }\; x\cdot \nu < 0\,.
			\end{cases}
		\end{equation*}
We notice that in this regime the passage from $\F_\e$ (or, equivalently, from \eqref{intro:eff-funct}) to the functionals \eqref{intro:sub}  can be made rigorous by combining the Modica-Mortola trick with a classical argument of Ambrosio based on the co-area formula. These two ingredients allow us to estimate from below the surface term in \eqref{intro:eff-funct}
with an $\e$-independent heterogeneous and anisotropic perimeter functional and then to conclude (see the proof of Proposition \ref{prop:supercritical_case}).     
		
If $\ell\in (0,+\infty)$, which corresponds to the case $\e \sim \delta_\e$, approximation and homogenisation procedure cannot be decoupled and $g^\ell_\hom$ is given by an asymptotic cell-formula in which the whole Modica-Mortola term in \eqref{intro:en} appears. More precisely, in this regime $g^\ell_\hom$ is given by
\begin{equation}\label{intro:critical}
\begin{split}
					g^\ell_{\rm hom}(\nu)= \lim_{r\to+\infty}\frac{1}{r^{n-1}}\inf\Biggl\{\int_{Q^\nu_r(0)}&\left((1-v)^2+h(\ell x,\nabla v)\right)\dx
					\colon \\
					& v\in W^{1,2}(Q^\nu_r(0))\,,0\le v\le1\colon\,\exists u\in W^{1,2}(Q_r^\nu(0);\R^m)\
					\text{with}\\& v\nabla u=0\ \text{a.e. in}\ Q_r^\nu(0) \ \text{and}\
					(u,v)=(\bar u^\nu,\bar v^\nu)\ \text{near}\ \partial Q^\nu_r(0)\Biggr\}\,,
				\end{split}  		
\end{equation} 
where $\bar u^\nu$ is a suitable regularisation of $u^\nu$ while $\bar v^\nu$ is equal to $1$ on most of the boundary of the cube $Q^\nu_r(0)$ and equal to $0$ in a neighbourhood of the hyperplane $x\cdot \nu=0$, and chosen in a way such that $\bar v^\nu \nabla \bar u^\nu=0$ a.e.\ in $Q_r^\nu(0)$ (see Section \ref{subs:notation}, \ref{uv-bar} for the definition of $(\bar u^\nu,\bar v^\nu)$ and \cf \cite[Theorem 3.5]{BMZ}). 
We notice that in view of the growth conditions satisfied by $f$ (see \ref{hyp:growth-f}), the first term in $\F_\e$ vanishes on those pairs $(u,v)$ which are admissible in \eqref{intro:critical}.  

Eventually, if $\ell=+\infty$, which corresponds to the case $\e \gg \delta_\e$, heuristically, we can first let $\delta_\e \to 0$ thus getting the spatially homogeneous functionals
\begin{equation}\label{intro:super}
 \int_A(v^2+\eta_\e)f_\hom(\nabla u)\dx + \frac{1}{\e} \int_A (1-v)^2 \dx + \e \int_A h_{\hom}(\nabla v)\dx,
 \end{equation}
 where 
 \begin{equation*}
				\begin{split}
					h_{\rm hom}(w)= \inf\Biggl\{\int_{(0,1)^n}h(x,\nabla v+w)\dx\colon v\in W_0^{1,2}((0,1)^n)\Biggr\}\,.
				\end{split}
\end{equation*}
Hence, letting $\e\to 0$ and invoking \cite{Foc01} give
\begin{equation*}
 \int_A f_\hom(\nabla u)\dx + 2 \int_{S_u \cap A} \sqrt{h_{\hom}(\nu)}\d \mathcal H^{n-1}
 \end{equation*}
in~\eqref{intro:G-lim}, so that in this case
\[
g^\infty_{\hom}(\nu)= 2 \sqrt{h_{\hom}(\nu)}. 
 \] 
 We notice that in this regime the most delicate step in the proof of the $\Gamma$-converge result is to show that $\F_\e$ actually behaves like \eqref{intro:super} or, in other words, that the passage to the limit as $\delta_\e \to 0$ can be rigorously justified.   
 To do so we show that we can replace the $v$-variables of a sequence $(u_\e,v_\e)$ with equi-bounded energy with a suitably averaged sequence $(\tilde v_\e)$ such that  
 \begin{equation}\label{intro:es-super}
 \frac{1}{\e} \int_A (1-v_\e)^2 \dx + \e \int_A h\Big(\frac{x}{\delta_\e},\nabla v_\e\Big)\dx \geq \frac{1}{\e} \int_A (1-\tilde v_\e)^2 \dx + \e \int_A h_{\hom}(\nabla \tilde v_\e)\dx+o(1).
  \end{equation}
 However, in order to ensure that after this replacement the corresponding volume term  
 \begin{equation}\label{intro:volume}       	
\int_A(\tilde v_\e^2+\eta_\e)f\Big(\frac{x}{\delta_\e},\nabla u_\e\Big)\dx,  
\end{equation}
or, equivalently,
\[
\int_A(\tilde v_\e^2+\eta_\e) f_\hom(\nabla u_\e)\dx, 
\]
remains uniformly bounded, we need some additional information on the blow-up rate of $|\nabla u_\e|$ which shall be related to the scale of oscillations $\delta_\e$ in the following way
\begin{equation}\label{intro:rate}
\int_A |\nabla u_\e|^2 \dx \sim \frac{1}{\delta_\e}.  
\end{equation}
The latter is then enforced by requiring that in the regime $\e \gg \delta_\e$ the infinitesimal parameter $\eta_\e$ appearing in \eqref{intro:en} is exactly of order $\delta_\e$ so that, thanks to the growth conditions satisfied by $f$, the assumption in \eqref{intro:rate} is automatically satisfied by the $u$-variables of any sequence $(u_\e,v_\e)$ with equi-bounded energy. 	

To conclude we would like to briefly comment on the additional assumption \eqref{intro:rate} which might be a drawback of our specific method of proof, inspired by~\cite{ABC01} (see the proof of Proposition \ref{prop:subcritical:case}). In fact, on the one hand the definition of the auxiliary sequence $(\tilde v_\e)$ is somehow dictated by the key estimate \eqref{intro:es-super} which, in its turn, is compatible with an analogous estimate for \eqref{intro:volume} only if also this term can be put in some relation with the oscillation parameter $\delta_\e$. On the other hand, from a modelling point of view, an assumption on the convergence rate to zero of $\eta_\e$ (so to enforce \eqref{intro:rate}) does not appear to be too restrictive.  We finally observe that also in the homogenisation of the Modica-Mortola functionals \cite{ABC01} the regime $\ell=+\infty$ is the most delicate one and requires an additional assumption on the vanishing rate of $\delta_\e$ as $\e \to 0$.

\subsection {Outline of the paper} In Section \ref{sect:setting} we set some notation, define the phase-field functionals $\F_\e$, and recall some preliminaries. Section \ref{sect:main} is devoted to the statement of the main $\Gamma$-convergence result, Theorem \ref{thm:main_theorem}, and to the proof of a convergence result for some associated minimisation problems, Corollary \ref{cor:conv-min}. Then the proof of Theorem \ref{thm:main_theorem} is carried over in a number of intermediate steps throughout sections \ref{sec:gamma-limit} - \ref{sec:finer}.  Namely, in Section \ref{sec:gamma-limit}, Theorem \ref{thm:gamma_limit} we prove that the functionals $\F_\e$ $\Gamma$-converge to a spatially homogeneous free-discontinuity functional and in Proposition \ref{prop:volume_density} we show that its volume energy density coincides with $f_\hom$. In Section \ref{sec:larger} we consider the regime $\ell=0$ (or equivalently $\e\ll \delta_\e$) and determine the homogenised surface energy density $g_\hom^0$ (see Proposition \ref{prop:supercritical_case}). Then, Section \ref{sec:same} is devoted to the characterisation of $g^\ell_\hom$ in the regime $\ell\in (0,+\infty)$ (or equivalently $\e\sim \delta_\e$) (see Proposition \ref{prop:critical_case}). Eventually, Section \ref{sec:finer} deals with the regime $\ell=+\infty$ (or equivalently $\e\gg \delta_\e$), where in this case to determine $g^\infty_\hom$ we make the additional technical assumption $\eta_\e \sim \delta_\e \sim \e^\alpha$, for some $\alpha>1$ (see Proposition \ref{prop:subcritical:case}).    

	%
	%
	\section{Setting of the problem and preliminary results}\label{sect:setting}
		
	In this section we introduce some useful notation, define the functionals under examination, and recall some preliminaries.

	\subsection{Notation}\label{subs:notation} We start collecting the notation we are going to employ throughout. 

	\begin{enumerate}[label=(\alph*)]
		\item $m,n \geq 1$ are fixed positive integers; we set $\R^m_0:=\R^m\setminus \{0\}$;
		\item $\S^{n-1}\defas\{\nu=(\nu_1,\ldots,\nu_n)\in \R^n \colon \nu_1^2+\cdots+\nu_n^2=1\}$ and $\widehat{\S}^{n-1}_\pm\defas\{\nu \in\S^{n-1}\colon \pm\nu_{i(\nu)}>0\}$, where $i(\nu)\defas\max\{i\in\{1,\ldots,n\}\colon \nu_i\neq 0\}$;
		\item $\mathcal L^n$ and and $\mathcal H^{n-1}$ denote, respectively, the Lebesgue measure and the $(n-1)$-dimensional Hausdorff measure on $\R^n$;
		\item $\A$ denotes the collection of all open and bounded subsets of $\R^n$ with Lipschitz boundary. If $A,B \in \A$ by $A \subset \subset B$ we mean that $A$ is relatively compact in $B$;
		\item $Q$ denotes the open unit cube in $\R^n$ with sides parallel to the coordinate axis, centred at the origin; for $x\in \R^n$ and $r>0$ we set $Q_r(x):= rQ+x$. If $x=0$ we simply write $Q_r$.
		Moreover, $Q'$ denotes the open unit cube in $\R^{n-1}$ with sides parallel to the coordinate axis, centred at the origin, for every $r>0$ we set $Q_r'\defas r Q'$;
		\item\label{Rn} for every $\nu\in \Sph^{n-1}$ let $R_\nu$ denote an orthogonal $(n\x n)$-matrix such that $R_\nu e_n=\nu$; we also assume that $R_{-\nu}Q=R_\nu Q$ for every $\nu \in \S^{n-1}$, $R_\nu\in\Q^{n\x n}$ if $\nu\in\S^{n-1}\cap\Q^n$, and that the restrictions of the map $\nu\mapsto R_\nu$ to $\widehat{\Sph}_{\pm}^{n-1}$ are continuous. For an explicit example of a map $\nu \mapsto R_\nu$ satisfying all these properties we refer the reader, \textit{e.g.}, to~\cite[Example A.1]{CDMSZ19};
		\item for $x\in\R^n$, $r>0$, and $\nu\in\S^{n-1}$, we define $Q^\nu_r(x):=R_\nu Q_r(0)+x$. If $x=0$ we simply write $Q^\nu_r$ and we set $Q^\nu\defas Q_1^\nu$;
		\item for $\xi\in \R^{m \x n}$ we let $u_\xi$ be the linear function whose gradient is equal to $\xi$\ie $u_\xi(x):=\xi x$, for every $x\in \R^n$;
		\item\label{jump-fun} for $\zeta\in \R^m_0$, and $\nu \in \Sph^{n-1}$ we denote with $u_{\zeta}^{\nu}$ the piecewise constant function taking values $0,\zeta$ and jumping across the hyperplane $\Pi^\nu:=\{x\in \R^n \colon x \cdot \nu=0\}$\ie
		\begin{equation*}
			u^{\nu}_{\zeta}(x):=\begin{cases} \zeta & \text{if }\; x\cdot \nu \geq 0\,, 
				\cr
				0 & \text{if }\; x\cdot \nu < 0\,,
			\end{cases}
		\end{equation*}
		when $\zeta=e_1$ we simply write $u^\nu$ in place of $u^{\nu}_{e_1}$;
		\item\label{1dim-couple} let ${\rm u} \in C^1(\R)$, ${\rm v}\in C^1(\R)$, with $0\leq \vv \leq 1$, be one-dimensional functions satisfying the following two properties:
		\smallskip
		
		\begin{enumerate}[label= \roman*.]
			\item $\vv\uu'\equiv 0$ in $\R$; 
			
			\smallskip
			
			\item $(\uu(t),\vv(t))=(\chi_{(0,+\infty)}(t),1)$ for $|t|>1$;
		\end{enumerate}
		
		\item\label{uv-bar} for $\nu \in \Sph^{n-1}$ we set
		\begin{equation*}
			\bar u^\nu(x):= \uu (x\cdot \nu) e_1\,, \quad \bar v^\nu(x):= \vv (x\cdot \nu)\,;
		\end{equation*}
		\item\label{uv-bar-e} for $\nu\in\S^{n-1}$, $\zeta\in\R^m_0$ and $\e>0$ we set
		\begin{equation*}
			\bar u^\nu_{\zeta,\e}(x)\defas \uu \big(\tfrac{1}{\e}x\cdot\nu)\zeta\,,\quad\bar v^\nu_{\e}(x)\defas\vv\big(\tfrac{1}{\e}x\cdot\nu).
		\end{equation*}
		When $\zeta=e_1$ we simply write $\bar{u}_{\e}^\nu$ in place of $\bar{u}_{e_1,\e}^\nu$. We notice that in particular, $\bar{u}_{1}^\nu=\bar{u}^\nu$, $\bar{v}_{1}^\nu=\bar{v}^\nu$;
	\end{enumerate}

	\medskip
	
	\noindent We now introduce the functional spaces relevant for our problem. Given a $\mathcal L^n$-measurable set $A\subset \R^n$ we let $L^0(A;\R^m)$ denote the space of all Lebesgue measurable functions mapping from $A$ to $\R^m$. On $L^0(A;\R^m)$ we consider the topology induced by the convergence in measure on bounded subsets of $A$. We recall that this topology is both metrisable and separable.
	
	For $A\subset\R^n$ open we consider the functional space $SBV(A;\R^m)$ (resp.\ $GSBV(A;\R^m)$) of special functions of bounded variation (resp.\ of generalised special functions of bounded variation) on $A$. We refer the reader to the monograph \cite{AFP} for the properties of those functional spaces; here we only recall that for any $u\in SBV(A;\R^m)$ the distributional derivative $Du$ is a bounded radon measure and can be represented as 
	\begin{equation}\label{c:SBV}
		Du(B)=\int_B \nabla u \dx+\int_{B\cap S_u}[u]\otimes \nu_u \d\mathcal H^{n-1},
	\end{equation}
	for every $B \in \mathcal B^n$, where $\mathcal B^n$ is the Borel $\sigma$- algebra of $\R^n$. In \eqref{c:SBV} $\nabla u$ denotes the approximate gradient of $u$ (which makes sense also for $u\in GSBV$), $S_u$ the set of approximate discontinuity points of $u$, $[u]:=u^+-u^-$ where $u^\pm$ are the one-sided approximate limit points of $u$ at $S_u$, and $\nu_u$ is the measure theoretic normal to $S_u$. 
	
	For $p>1$ we also consider the functional spaces
	\begin{equation*}
		SBV^{p}(A;\R^m):= \{u \in SBV(A;\R^m)\colon \nabla u\in L^{p}(A;\R^{m\x n}) \text{ and } \mathcal{H}^{n-1}(S_{u})<+\infty\}\,,
	\end{equation*}
	and
	\begin{equation*}
		GSBV^{p}(A;\R^m):= \{u \in GSBV(A;\R^m)\colon \nabla u\in L^{p}(A;\R^{m\x n}) \text{ and } \mathcal{H}^{n-1}(S_{u})<+\infty\}\,.
	\end{equation*}
	We recall that $GSBV^p(A;\R^m)$ is a vector space; moreover, if $u\in GSBV^p(A;\R^m)$ then we have that $\phi(u) \in SBV^p(A;\R^m)\cap L^\infty(A;\R^m)$, for every $\phi \in C^1(\R^m;\R^m)$ with ${\rm supp} (\nabla \phi) \subset \subset \R^m$ (see \cite{DMFT}). 
	
	Eventually, we say that a function $h:\R^n\to\R^m$ is $r$-periodic for some $r>0$, if $h(x+r e_i)=h(x)$ for every $i\in\{1,\ldots,n\}$.
	
	\medskip
	
	Throughout the paper $C$ denotes a strictly positive constant which may vary from line to line and within the same expression.

	\subsection{Setting of the problem} 
	In this subsection we introduce the functionals we are going to analyse in this paper. To this end, let $f:\R^n\times\R^{m\times n}\to[0,+\infty)$ and $h:\R^n\times\R^n\to[0,+\infty)$ be Borel measurable functions satisfying, respectively, the following hypotheses:
	\begin{enumerate}[label=($f\arabic*$)]
		\item\label{hyp:growth-f}(growth conditions) there exist two constants $0<c_1\le c_2<+\infty$ such that for every $x \in \R^n$ and every $\xi \in \R^{m\x n}$
		\begin{equation*}
			c_1 |\xi|^2 \leq f(x,\xi)\leq c_2 |\xi|^2\, ;
		\end{equation*}
		\item\label{hyp:cont-xi-f} (continuity in $\xi$) there exists $0<L_1<+\infty$ such that for every $x \in \R^{n}$ we have
		\begin{equation*}
			|f(x,\xi_1)-f(x,\xi_2)| \leq L_1\big(1+|\xi_1|+|\xi_2|\big)|\xi_1-\xi_2|,
		\end{equation*}
		for every $\xi_1$, $\xi_2 \in \R^{m\x n}$;
		\item\label{hyp:per-f}(periodicity in $x$) for all $\xi\in\R^{m\times n}$, $f(\cdot,\xi)$ is $Q$-periodic;
	\end{enumerate}
	
	\medskip
	\begin{enumerate}[label=($h\arabic*$)]
		\item\label{hyp:growth-h} (growth conditions) there exist two constants $0<c_3\le c_4<+\infty$ such that for every $x \in \R^n$, and every $w \in \R^{n}$
		\begin{equation*}
			c_3|w|^2 \leq h(x,w)\leq c_4|w|^2\, ;
		\end{equation*}
		\item\label{hyp:cont-h} (continuity in $w$) there exists $0<L_2<+\infty$ such that for every $x \in \R^n$ we have
		\begin{equation*}
			|h(x,w_1)-h(x,w_2)| \leq L_2 \big(1+|w_1|+|w_2|\big)|w_1-w_2|
		\end{equation*}
		for every $w_1$, $w_2 \in \R^{n}$;
		\item\label{hyp:hom-h} (homogeneity in $w$) for all $x\in\R^n$, $h(x,\cdot)$ is homogeneous of degree two; i.e., 
		\begin{equation*}
			h(x,s w)=s^2h(x,w)
		\end{equation*}
		for all $w\in\R^n$, $s\in\R$;
		\item\label{hyp:Lip-h}(Lipschitz-continuity in $x$) there exists $0<L_3<+\infty$ such that for every $w\in\R^n$ we have
			\begin{equation*}
			|h(x_1,w)-h(x_2,w)| \leq L_3 |x_1-x_2|
		\end{equation*}
		for every $x_1$, $x_2 \in \R^{n}$;
		\item\label{hyp:per-h} (periodicity in $x$) for all $w\in\R^n$, $h(\cdot,w)$ is $Q$-periodic.
	\end{enumerate}
	In all that follows $\e>0$ varies in a family of strictly positive parameters converging to zero and $\delta_\eps>0$ is a strictly increasing function of $\e$ with $\delta_\e\searrow 0$ as $\e\searrow 0$. Set
		\begin{equation}\label{def:limit_l}
			\ell\defas\lim_{\e\to0}\frac{\e}{\delta_\e}\ \in [0,+\infty].
		\end{equation}
		Moreover, throughout the paper we let $0<\eta_\e\ll\e$.
		
		\smallskip
	For given Borel integrands $f:\R^n\times\R^{m\times n}\to [0,+\infty)$ and $h:\R^n\times\R^n\to [0,+\infty)$ as above, we introduce the functionals $\F_\e\colon L^0(\R^n;\R^m) \times L^0(\R^n) \times \A \longrightarrow [0,+\infty]$ defined by
		\begin{align}\label{F_e}
			\F_\e(u,v,A) &\defas 
			\begin{cases}
			\displaystyle\int_A(v^2+\eta_\e)f\Big(\frac{x}{\delta_\e},\nabla u\Big)\dx &+\displaystyle\int_A\left(\frac{(1-v)^2}{\e}+\e h\Big(\frac{x}{\delta_\e},\nabla v\Big)\right)\dx
			\\[2em]
			&\text{if}\ (u,v)\in W^{1,2}(A;\R^m)\x W^{1,2}(A)\, ,\ 0\le v\le1\,,\\[1em]
			+\infty &\text{otherwise in}\ \LtL\,.
			\end{cases}
		\end{align}
		It is convenient to introduce a notation for the regularised surface term in $\F_\e$, $\F_\e^s:L^0(\R^n)\times\A\longrightarrow [0,+\infty]$; we set
		\begin{align*}\label{F-surf}
		\F_\e^s(v,A)\defas
		\begin{cases}
		\displaystyle\int_A\left(\frac{(1-v)^2}{\e}+\e h\Big(\frac{x}{\delta_\e},\nabla v\Big)\right)\dx &\text{if}\ v\in W^{1,2}(A)\,,\ 0\leq v\leq 1\,,\\[2em]
		+\infty &\text{otherwise in}\ L^0(\R^n)\,.
		\end{cases}
		\end{align*}
		\begin{remark}\label{rem:continuity-functionals}
		The following observations are in order. 
		\begin{enumerate}
\item In view of hypotheses \ref{hyp:growth-f}--\ref{hyp:cont-xi-f} and~\ref{hyp:growth-h}--\ref{hyp:hom-h}, for $\eta_\e\equiv 0$ the functionals $\F_\e$ in~\eqref{F_e} belong to the class of functionals introduced and analysed in~\cite{BMZ}. Moreover, a one-dimensional variant of $\F_\e$ has been analysed by the authors in \cite{BEMZ}.

\item The assumptions on $f$ and $h$ ensure, in particular, that for every $A\in\A$ the functionals $\F_\e(\cdot,\cdot,A)$ are continuous in the strong $W^{1,2}(A;\R^m)\times W^{1,2}(A)$ topology.

\item Assumptions~\ref{hyp:growth-f} and \ref{hyp:growth-h} imply that for every $A\in\A$ and every $(u,v)\in W^{1,2}(A;\R^m)\times W^{1,2}(A)$, $0\leq v\leq 1$ there holds
	\begin{equation}\label{est:bounds-Fe}
		\begin{split}
			\min\{1\,,c_1\, , c_3\}AT_\e(u,v)
			\leq\F_\e(u,v,A)
			\leq	\max\{1\,,c_2\, , c_4\}AT_\e(u,v)\, ,
		\end{split}
	\end{equation}
	where 
	\begin{equation}\label{def:AT}
		AT_\e(u,v)\defas 	\int_A (v^2+\eta_\e)|\nabla u|^2\dx +\int_A\bigg(\frac{(1-v)^2}{\e}+\e|\nabla v|^2\bigg)\dx\,
	\end{equation}
	is the Ambrosio-Tortorelli functional \cite{AT92}. 
	\end{enumerate}
\end{remark}
	\subsection{Preliminary remarks on the Ambrosio-Tortorelli functional}
	We close this first section by recalling some results on the convergence of suitable variants of the Ambrosio-Tortorelli functional above together with some properties of the so-called optimal profile problem.
	\begin{remark}\label{rem:AT}
	By virtue of~\cite{AT92} we know that the functionals $AT_\e$, defined in~\eqref{def:AT}, $\Gamma$-converge in $\LtL$ to the Mumford-Shah functional
	\begin{equation*}
	MS(u,1)=\int_A|\nabla u|^2\dx+2\H^{n-1}(S_u\cap A)\quad u\in GSBV^2(A;\R^m)\,.
	\end{equation*} 
	Moreover, \cite[Theorem 3.1]{Foc01} states that for any $p>1$, any $a,b>0$, and any norm $\varphi:\R^n\to[0,+\infty)$ the anisotropic Ambrosio-Tortorelli functionals
	\begin{equation}\label{AT-anisotropic}
	\E_\e(u,v)\defas a\int_A (v^p+\eta_\e)|\nabla u|^p\dx+b\int_A\bigg(\frac{(1-v)^2}{\e}+\e \varphi^2(\nabla v)\bigg)\dx,
	\end{equation}
	with $(u,v)\in W^{1,p}(A;\R^m)\times W^{1,2}(A)$ and $0\leq v\leq 1$, $\Gamma$-converge in $\LtL$ to the anisotropic free-discontinuity functional
	\begin{equation}\label{MS-anisotropic}
	\E(u,1)=a\int_A|\nabla u|^p\dx+2b\int_{S_u\cap A}\varphi(\nu_u)\dHn\quad u\in GSBV^p(A;\R^m)\,.
	\end{equation}
	Although~\cite[Theorem 3.1]{Foc01} is stated in the case where in $\E_\e$ the functions $\nabla u$ and $v$ have the same summability exponent $p>1$, an inspection of the proof reveals that different exponents can be also considered (\cf also~\cite[Theorem 5.1]{ChamCris-Density}). 
	\end{remark}
	\begin{remark}\label{rem:op}
Let $\lambda>0$; arguing as in, \eg~\cite[Chapter 6]{B}, it is immediate to check that
\begin{align}
\sqrt{\lambda} &= \min\bigg\{\int_0^{+\infty}\hspace{-0.3cm}\big((1-v)^2 + \lambda\, (v')^2 \big)\dx \colon v \in W^{1,2}_{\rm loc}(0,+\infty),\; 0\leq v \leq 1,\;
v(0)=0,\; v(+\infty)=1\bigg\}\label{e:op-AT-1} 
\\[2pt]
&= \inf_{T>0}\min\bigg\{\int_0^T\hspace{-0.2cm}\big((1-v)^2 + \lambda\, (v')^2 \big)\dx \colon v \in W^{1,2}(0,T),\; 0\leq v \leq 1,\;
v(0)=0,\; v(T)=1\bigg\}\,,\label{e:op-AT-2} 
\end{align}
where $v(+\infty)\defas\lim_{t\to+\infty}v(t)$. Indeed, a solution to the minimisation problem in~\eqref{e:op-AT-1} is given by the smooth function $v_\lambda(t)=1-\exp(-t/\sqrt{\lambda})$. Then, for every $T>0$ a competitor for the minimisation problem in~\eqref{e:op-AT-2} can be obtained by linearly interpolating on $(T-1,T)$ between $v_\lambda(T-1)$ and $1$, thus approaching the value $\sqrt{\lambda}$, as $T\to+\infty$. In particular, the minimisation in~\eqref{e:op-AT-2} can be carried over all Lipschitz continuous functions $v$ satisfying the same boundary conditions.
\end{remark}


\section{Statement of the main result}\label{sect:main} In this section we state the main result of this paper, Theorem~\ref{thm:main_theorem}. The latter establishes the $\Gamma$-convergence of $\F_\e$ as $\e\to 0$ in the three regimes $\e\ll\delta_\e$, $\e\sim\delta_\e$, and $\e\gg \delta_\e$. As a corollary of Theorem~\ref{thm:main_theorem}, we then prove the convergence of some minimisation problems associated to $\F_\e$. 
	\begin{theorem}[$\Gamma$-convergence]\label{thm:main_theorem}
		Let $\F_\e$ and $\ell$ be as in~\eqref{F_e} and~\eqref{def:limit_l}, respectively. If $\ell=+\infty$, assume moreover that $\eta_\e \simeq \delta_\e \simeq\e^\alpha$, for some $\alpha >1$. Then for every $A\in\mathcal{A}$ the functionals $\F_\e(\cdot,\cdot,A)$ $\Gamma$-converge in $L^0(\R^n;\R^m)\times L^0(\R^n)$ to the homogeneous functional $\F_{\rm hom}^\ell(\cdot,\cdot,A)$, where $\F_{\rm hom}^\ell\colon L^0(\R^n;\R^m)\times L^0(\R^n)\times\mathcal{A}\longrightarrow [0,+\infty]$ is given by
		\begin{equation}\label{F_hom}
			\F_{\rm hom}^\ell(u,v,A):=
			\begin{cases}
				\displaystyle\int_Af_{\rm hom}(\nabla u)\dx+\int_{S_u\cap A}g_{\rm hom}^{\ell}(\nu_u)\dHn&\text{if}\ u\in GSBV^2(A;\R^m)\, ,\\ & v= 1\ \text{a.e. in}\ A \, ,\\
				+\infty &\text{otherwise}\,,
			\end{cases}
		\end{equation}
		with $f_{\rm hom}\colon\R^{m\times n}\to[0,+\infty)$ and $g^\ell_{\rm hom}\colon\S^{n-1}\to[0,+\infty)$ Borel functions. 
		
		Moreover, for every $\xi\in \R^{m\times n}$ there holds
		\begin{equation}\label{f_hom}
			f_{\rm hom}(\xi)\defas \lim_{r\to+\infty}\frac{1}{r^n}\inf\Biggl\{\int_{Q_r}f(x,\nabla u)\dx\colon u\in W^{1,2}(Q_r;\R^m)\,,u=u_\xi\ \text{near}\ \partial Q_r \Biggr\}. 
		\end{equation}
For every $\nu\in\S^{n-1}$ we have:
		\begin{enumerate}[label=  $(\roman*)$]
			\item if $\ell=0$, then 
			\begin{equation}\label{g_hom_supercritical}
				\begin{split}
					g^0_{\rm hom}(\nu)\defas 2 \lim_{r\to+\infty}\frac{1}{r^{n-1}}\inf\Biggl\{\int_{S_u\cap Q^\nu_r}&\sqrt{h(x,\nu_u)}\dHn\colon \\
					&u\in BV(Q_r^\nu;\{0,e_1\})\,, u=u^\nu\ \text{near}\ \partial Q^\nu_r\Biggr\}\,;
				\end{split}
			\end{equation}
			\item  if $\ell\in(0,+\infty)$, then 
			\begin{equation}\label{g_hom_critical}
				\begin{split}
					g^\ell_{\rm hom}(\nu)\defas \lim_{r\to+\infty}\frac{1}{r^{n-1}}\inf\Biggl\{\int_{Q^\nu_r}&\left((1-v)^2+h(\ell x,\nabla v)\right)\dx
					\colon \\
					& v\in W^{1,2}(Q^\nu_r)\,,0\le v\le1\colon\,\exists u\in W^{1,2}(Q_r^\nu;\R^m)\
					\text{with}\\& v\nabla u=0\ \text{a.e. in}\ Q_r^\nu \ \text{and}\
					(u,v)=(\bar u^\nu,\bar v^\nu)\ \text{near}\ \partial Q^\nu_r\Biggr\}\,;
				\end{split}
			\end{equation}
		\item if $\ell=+\infty$, then
		\begin{equation}\label{g_hom_subcritical}
				g^{\infty}_{\rm hom}(\nu)\defas 2\sqrt{h_{\rm hom}(\nu)}\,,
			\end{equation}
			where $h_\hom:\R^n\to[0,+\infty)$ is given by
			\begin{equation}\label{h_hom_subcritical}
				\begin{split}
					h_{\rm hom}(w)\defas \inf\Biggl\{\int_{Q}h(x,\nabla v+w)\dx\colon v\in W_0^{1,2}(Q)\Biggr\}\,.
				\end{split}
			\end{equation}
			\end{enumerate}
	\end{theorem}	

	\begin{remark}[Properties of $f_\hom$]\label{rem:prop-f}
	The homogenised bulk integrand $f_\hom$ in~\eqref{f_hom} coincides with the bulk integrand obtained in~\cite{Braides85, Mueller87, BDV}. In particular, the limit in~\eqref{f_hom} exists and $f_\hom$ is quasiconvex.  Moreover, $f_\hom$ can be rewritten as
	\begin{equation}\label{f-hom-equiv}
	f_\hom(\xi)=\inf_{r\in \mathbb N^*}\frac{1}{r^n}\inf\Biggl\{\int_{Q_r}f(x,\nabla u+\xi)\dx\colon u\in W^{1,2}(Q_r;\R^m)\,,u=0\ \text{near}\ \partial Q_r \Biggr\}\,.
	\end{equation}
	\end{remark}
	\begin{remark}[Properties of $g_\hom^\ell$]\label{rem:prop-g}
	Some observations on the surface integrand $g_\hom^\ell$ are in order.
	\begin{enumerate}
	\item\label{rem:g-0} In the regime $\ell=0$ the surface integrand $g_\hom^0$ defined in~\eqref{g_hom_supercritical} coincides with the one obtained in the homogenisation of functionals defined on finite partitions~\cite{AB90-2}. More precisely,~\cite[Proposition 4.4]{AB90-2} ensures that $g_\hom^0(\nu)$ is well defined for every $\nu\in\S^{n-1}$. In addition,~\cite[Theorem 4.2 and Example 2.8]{AB90-2} ensure that the $1$-homogeneous extension of $g_\hom^0$ to $\R^n$ is convex, and therefore continuous. Moreover, in view of~\cite[Theorem 3.1 and Example 2.8]{AB90-2} the value of $g_\hom^0(\nu)$ remains unchanged if the surface integrand $\sqrt{h}$ in~\eqref{g_hom_supercritical} is replaced by its convex envelope with respect to the second variable. 
	\item\label{rem:g-ell} For $\ell=1$, the existence of the limit defining $g_\hom^\ell(\nu)$ as well as the continuity of $g_\hom^\ell$ restricted to $\widehat{\S}_{\pm}^{n-1}$ is established in~\cite[Proposition 8.7]{BMZ} in a more general setting. The arguments as above can be used to show that the same continuity properties hold true for any $\ell\in(0,+\infty)$. 
	\item\label{rem:g-infty} In the regime $\ell=+\infty$ the function $g_\hom^\infty$ is easily seen to be continuous. 
Indeed it is known that the function $h_\hom$ satisfies the growth condition~\ref{hyp:growth-h} and the local Lipschitz condition~\ref{hyp:cont-h}, albeit with a different constant $L_2'$ (\cf \cite[Lemma 2.1]{Mueller87}). Moreover, $h_\hom$ is convex. 
	\end{enumerate}
	\end{remark}
	
The proof of Theorem~\ref{thm:main_theorem} will be divided into a number of intermediate steps and will be carried out in Sections~\ref{sec:gamma-limit}--\ref{sec:finer}. Specifically, in Section~\ref{sec:gamma-limit} we show that there exists a subsequence $(\e_k)$ such that for every $A\in\A$ the corresponding functionals $\F_{\e_k}(\cdot,\cdot,A)$ $\Gamma$-converge to a free-discontinuity functional which is finite on $GSBV^2(A;\R^m)\times\{1\}$ and of the form
	\begin{equation*}
	\int_A{f}^\ell(\nabla u)\dx+\int_{S_u\cap A}{g}^\ell([u],\nu_u)\dHn\,.
	\end{equation*}	
	At this stage the integrands ${f}^\ell$ and ${g}^\ell$ may a priori depend on the subsequence $(\e_k)$. 
	The procedure followed to prove such a compactness and integral representation result is by now classical, moreover the corresponding result for $\eta_\e\equiv 0$ can be found in~\cite[Theorem 5.2]{BMZ}. For these reasons we will only sketch this proof here (see Theorem \ref{thm:gamma_limit}). Then,  in view of~\cite[Theorem 5.2 and Theorem 3.1]{BMZ} we show that the volume integrand ${f}^\ell$ coincides with $f_\hom$ given by~\eqref{f_hom}, and therefore it is independent of $\ell$ and $(\e_k)$. Eventually, in Sections~\ref{sec:larger}, \ref{sec:same}, and~\ref{sec:finer} we characterise ${g}^\ell$ in the three regimes $\ell=0$, $\ell\in(0,+\infty)$, and $\ell=+\infty$, respectively. Namely, we show that ${g}^\ell=g_\hom^\ell$, where we notice that the latter does not depend on the subsequence $(\e_k)$. Consequently, Theorem~\ref{thm:main_theorem} follows by the Urysohn property of $\Gamma$-convergence~\cite[Proposition 8.3]{DM-book}.  
	
	\medskip
	
	To conclude, we also observe that hypothesis~\ref{hyp:Lip-h} will be used only in the proof of Proposition \ref{prop:supercritical_case}. In particular, for $\ell\in (0,+\infty]$ Theorem~\ref{thm:main_theorem} holds true without assuming any continuity of $h$ in $x$.	
	
	\medskip
	
	On account of Theorem~\ref{thm:main_theorem} we can prove the following convergence result for a class of minimisation problems associated to $\F_\e$.
	\begin{corollary}\label{cor:conv-min}
	Assume the hypotheses of Theorem~\ref{thm:main_theorem} are satisfied and assume in addition that for a.e.~$x\in\R^n$ the functions $f(x,\cdot)$ and $h(x,\cdot)$ are quasiconvex and convex, respectively. Let $A\in\A$, $q\geq 1$, and $g\in L^q(A;\R^m)$. Then
\begin{itemize}	
	\item for any $\e>0$ there exists a solution $(\bar{u}_\e,\bar{v}_\e)\in W^{1,2}(A;\R^m)\times W^{1,2}(A)$ to the minimisation  problem
	\begin{equation}\label{def:Mk}
	M_\e\defas\min\bigg\{\F_\e(u,v,A)+\int_A|u-g|^q\dx\colon (u,v)\in \LtL\bigg\}\,;
	\end{equation}
	\item $\bar{v}_\e\to 1$ in $L^2(A)$, as $\e\to 0$; 
	\item up to subsequences, $(\bar{u}_\e)$ converges in $L^q(A;\R^m)$ to a solution of
	\begin{equation}\label{def:M}
	M^\ell\defas\min\bigg\{\F_\hom^{\ell}(u,1,A)+\int_A|u-g|^q\dx\colon u\in GSBV^2(A;\R^m)\cap L^q(A;\R^m)\bigg\}\,;
	\end{equation}
	\item $M_\e\to M^\ell$, as $\e\to 0$.
	\end{itemize}
	\end{corollary}
	\begin{proof}
	For fixed $\e>0$ the existence of  a minimizing pair $(\bar{u}_\e,\bar{v}_\e)\in W^{1,2}(A;\R^m)\times W^{1,2}(A)$ for \eqref{def:Mk} follows by a straightforward application of the direct method of the calculus of variations. 
	
	The convergence $\bar v_\e\to 1$ in $L^2(A)$ readily follows by the definition of $\F_\e$. Moreover, by \cite[Lemma~4.1]{Foc01} up to subsequences (not relabelled) $\bar u_{\e}\to\bar u^\ell$ in $L^0(A;\R^m)$, for some $\bar u^\ell\in GSBV^2(A;\R^m)$. Eventually arguing as in \cite[Theorem~7.1]{DMI} we deduce that $M_\e\to M^\ell$, $\bar u_\e\to \bar u^\ell$ in $L^q(A;\R^m)$, and that $\bar u^\ell$ is a solution to \eqref{def:M}.
	\end{proof}

	\section{An abstract $\Gamma$-convergence result}\label{sec:gamma-limit}
	In this section we prove an abstract $\Gamma$-convergence result for the functionals $\F_\e$. We notice that if in~\eqref{F_e} we choose $\eta_\e\equiv 0$ then the functionals $\F_\e$ are a special instance of those considered in~\cite{BMZ} for which a $\Gamma$-convergence and integral representation result was established (cf. \cite[Theorem 5.2]{BMZ}). Since in this case we would like to allow for the presence of the infinitesimal sequence $\eta_\e$, with $0<\eta_\e\ll\e$, we need to show that the analogue of~\cite[Theorem 5.2]{BMZ} holds true in this case as well. The proof of this result will be very close to that of~\cite[Theorem 5.2]{BMZ}, for this reason we will only sketch it here, referring to~\cite{BMZ} for the details.    
	
	\begin{remark}\label{rem:fund-est}
	A crucial step in the proof of the $\Gamma$-convergence result below is to show that the functionals $\F_\e$ in~\eqref{F_e} satisfy a so-called fundamental estimate, uniformly in $\e$. We observe that such an estimate easily follows arguing as in~\cite[Proposition 5.1]{BMZ}. Indeed, thanks to~\ref{hyp:growth-f}, for every $A\in\A$ and every $u\in W^{1,2}(A;\R^m)$ the term $\eta_\e\int_Af\big(\frac{x}{\delta_\e},\nabla u\big)\dx$ can be bounded (up to a multiplicative constant) from above and from below by the convex term $\eta_\e \int_A|\nabla u|^2\dx$ to which the construction in~\cite[Proposition 5.1]{BMZ} directly applies.  
	\end{remark}
	\begin{theorem}\label{thm:gamma_limit}
		Let $\F_\e$ be as in~\eqref{F_e}; then there exists a subsequence $(\e_k)$ such that for every $A\in\A$ the functionals $\F_{\e_k}(\cdot,\cdot,A)$ $\Gamma$-converge in $L^0(\R^n;\R^m)\times L^0(\R^n)$ to $\F^\ell(\cdot,\cdot,A)$, where $\F^\ell\colon L^0(\R^n;\R^m)\times L^0(\R^n)\times\A\longrightarrow [0,+\infty]$ is given by
	\begin{equation}\label{tilde-F}
	\F^\ell(u, v, A)\defas
		\begin{cases}
			\displaystyle\int_A f^\ell(\nabla u)\dx+\int_{S_u \cap A} g^\ell([u],\nu_u)\dHn &\text{if}\ u \in GSBV^2(A;\R^m)\, ,\\
			& v= 1\, \text{a.e.\ in } A\, ,\\[4pt]
			+\infty &\text{otherwise}\,,
		\end{cases}
	\end{equation}
	 for some Borel functions $f^\ell \colon \R^{m \times n} \to [0,+\infty)$ and $g^\ell\colon\R^m\times  \Sph^{n-1} \to [0,+\infty)$.
\end{theorem}
\begin{proof}
	Since $0<\eta_\e\ll\e$, thanks to~\eqref{est:bounds-Fe} and Remark~\ref{rem:AT} we deduce the existence of a constant $C>0$ such that
	\begin{equation*}
		\begin{split}
			\frac{1}{C}\Big(\int_A|\nabla u|^2\dx+\H^{n-1}(S_u\cap A)\Big) &\leq{(\F^\ell)}'(u,1,A)\\
			&\leq(\F^\ell)''(u,1,A)\leq C\Big(\int_A|\nabla u|^2\dx+\H^{n-1}(S_u\cap A)\Big)\,,
		\end{split}
	\end{equation*}
	for every $A\in\A$ and $u\in GSBV^2(A;\R^m)$, where  ${(\F^\ell)}'(\cdot,\cdot,A)$ and ${(\F^\ell)}''(\cdot,\cdot,A)$ denote the $\Gamma$-liminf and $\Gamma$-limsup of $\F_\e(\cdot,\cdot,A)$, respectively.
	Moreover thanks to Remark~\ref{rem:fund-est} the functionals $\F_\e$ satisfy the fundamental estimate~\cite[Proposition 5.1]{BMZ}. Thus, arguing as in~\cite[Theorem 5.2]{BMZ} we can apply the localisation method of $\Gamma$-convergence (see \eg~\cite[Chapters 14--18]{DM-book}) together with the integral-representation result~\cite[Theorem 1]{BFLM} to deduce the existence of a subsequence $(\F_{\e_k})$ and a functional $\F^\ell:\LtL\times\A\longrightarrow [0,+\infty]$ with the following properties: For every $A\in\A$ the functionals $\F_{\e_k}(\cdot,\cdot,A)$ $\Gamma$-converge in $\LtL$ to $\F^\ell(\cdot,\cdot, A)$ and for every $u\in GSBV^2(A;\R^m)$ there holds
	\begin{equation*}
	\F^\ell(u,1,A)=\int_{A}f^\ell(x,\nabla u)\dx+\int_{S_u\cap A} g^\ell(x,[u],\nu_u)\dHn\,,
	\end{equation*}
	for some Borel functions $f^\ell:\R^n\times\R^{m\times n}\to [0,+\infty)$, $g^\ell:\R^n\times\R_0^m\times\R^n\to [0,+\infty)$, while $\F^\ell(\cdot,\cdot,A)=+\infty$ if either $u\not\in GSBV^2(A;\R^m)$ or $v\neq 1$. Eventually, thanks to~\ref{hyp:per-f} and~\ref{hyp:per-h} a well-known argument (see, \eg \cite[Lemma 3.7 (ii)]{BDV}) shows that $\F^\ell$ is invariant under translation in $x$. This in turn implies that $f^\ell$ and $g^\ell$ are independent of $x$, hence the claim follows.
\end{proof}
\begin{remark}\label{rem:prop-g-tilde}
By the general properties of $\Gamma$-convergence, for every $A\in\A$ the functional $\F^\ell(\cdot,1,A)$ in~\eqref{tilde-F} is lower semicontinuous with respect to the convergence in measure. In particular, the functional $u\mapsto\int_{S_u\cap A}{g}^\ell([u],\nu_u)\dHn$ is lower semicontinuous on finite partitions. As a consequence (see~\cite{AB90-2}), we deduce that for every $\nu\in\S^{n-1}$ the function ${g}^\ell(\cdot,\nu)$ is subadditive, while for every $\zeta\in\R_0^m$ the $1$-homogeneous extension of ${g}^\ell(\zeta,\cdot)$ is convex. In particular, ${g}^\ell(\zeta,\cdot)$ is continuous.
\end{remark}
By virtue of Theorem~\ref{thm:gamma_limit}, to complete the proof of Theorem~\ref{thm:main_theorem} we need to characterise the integrands ${f}^\ell$ and ${g}^\ell$ in~\eqref{tilde-F}, for every $\ell \in [0,+\infty]$. To this end, we preliminarily compare them to the bulk and surface integrands obtained by applying~\cite[Theorem 3.1]{BMZ} to $\F_\e$ when $\eta_\e\equiv 0$. We start recalling some of the notation employed in~\cite{BMZ}.

\medskip
For $\e>0$, $\rho>2\eps$, $\xi\in\R^{m\times n}$, and $\nu\in\S^{n-1}$ we define the two following minimisation problems
	\begin{equation*}\label{eq:m^b}
		\m^{b}_\e(u_\xi,Q_\rho)\defas\inf\bigg\{\int_{Q_\rho}f\Big(\frac{x}{\delta_\e},\nabla u\Big)\dx
		\colon u\in W^{1,2}(Q_\rho;\R^m)\,, u=u_\xi\ \text{near}\ \partial Q_\rho\bigg\}\,,
	\end{equation*}
	and
	\begin{equation*}\label{eq:m^s}
		\m^{s}_\e(\bar u_{\e}^\nu,Q^\nu_\rho)\defas\inf\bigg\{\int_{Q_\rho^\nu}\left(\frac{(1-v)^2}{\e}+\e h\Big(\frac{x}{\delta_\e},\nabla v\Big)\right)\dx
		\colon v\in\mathscr A(\overline u_{\e}^\nu,Q_\rho^\nu)\bigg\}\,,
	\end{equation*}
	with
	\begin{equation}\label{def_A}
	\begin{split}
			\mathscr A(\bar u_{\e}^\nu,Q_\rho^\nu) \defas\big\{v\in W^{1,2}(Q^\nu_\rho),\ 0\le v\le 1\colon\exists\, u\in W^{1,2}(Q^\nu_\rho;\R^m)\ &\text{with}\ v\nabla u=0\ \text{a.e. in}\ Q^\nu_\rho\\
			&\text{and}\ (u,v)=(\bar u_{\e}^\nu,\bar v_{\e}^\nu)\ \text{near}\ \partial Q^\nu_\rho\big\}
			\end{split}
	\end{equation}
	where $(\bar{u}_{\e}^\nu,\bar{v}_{\e}^\nu)$ is as in~\ref{uv-bar-e}.
	\begin{remark}\label{rem:f-g-tilde}
		If $\eta_\e\equiv 0$, by invoking~\cite[Theorem 3.1]{BMZ} we can deduce the existence of a sequence $(\e_k)$ such that the corresponding functionals $\F_{\e_k}(\cdot,\cdot,A)$ $\Gamma$-converge in $\LtL$ to 
		\begin{equation*}
		\hat F^\ell(u,1,A)=\int_A\hat{f}^\ell(\nabla u)\dx+\int_{S_u\cap A}\hat{g}^\ell(\nu_u)\dHn\quad u\in GSBV^2(A;\R^m)\,,
		\end{equation*}
		where the integrands $\hat{f}^\ell$ and $\hat g^\ell$ are given by
		\begin{equation}\label{eq:densities-vol}
				\begin{split}
				\hat	f^\ell(\xi)&=
					\limsup_{\rho\to0}\frac{1}{\rho^n}\limsup_{k\to\infty}\m^{b}_{\e_k}(u_\xi,Q_\rho)=\limsup_{\rho\to0}\frac{1}{\rho^n}\liminf_{k\to\infty}\m^{b}_{\e_k}(u_\xi,Q_\rho)
				\end{split}
			\end{equation}
			and
			\begin{equation}\label{eq:densities-sur}
				\begin{split}
				\hat	g^\ell(\nu)&= \limsup_{\rho\to0}\frac{1}{\rho^{n-1}}\limsup_{k\to\infty}\m^{s}_{\e_k}( \bar{u}_{\e_k}^\nu,Q^\nu_\rho)= \limsup_{\rho\to0}\frac{1}{\rho^{n-1}}\liminf_{k\to\infty}\m^{s}_{\e_k}(\bar u^\nu_{\e_k},Q^\nu_\rho)\,.
				\end{split}
			\end{equation}
For later use we observe that arguing as in~\cite[Proposition 2.6]{BMZ} one can show that $\hat{g}^\ell(\nu)$ in~\eqref{eq:densities-sur} can be equivalently written by replacing $\m_{\e_k}^s(\bar u_{\e_k}^\nu,Q_\rho^\nu)$ with $\m_{\e_k}^s(\bar{u}_{\alpha _k}^\nu,Q_\rho^\nu)$, with $\alpha_k\sim\e_k$. Moreover, an inspection of the proof of~\cite[Proposition 7.4]{BMZ} reveals that the $u$-variable in~\eqref{def_A} can be taken such that $\|u\|_{L^\infty}\leq 1$ and 
	\begin{equation}\label{cond:u-variable}
	u(x)\in\{0,e_1\}\quad\text{if}\quad v(x)\neq 0,
	\end{equation}
	for a.e. $x\in Q_\rho^\nu$.

In view of \eqref{eq:densities-sur} we know that for every $\ell\in [0,+\infty]$ the surface integrand $\hat{g}^\ell$ is independent of the jump opening $[u]$.
	We notice, however, that the methods employed in~\cite{BMZ} to establish the independence of $[u]$ of the surface term cannot be directly transferred to the case $\eta_\e>0$ and therefore to the integrand ${g}^\ell$ appearing in~\eqref{tilde-F}. 
\end{remark}	
	
	Upon assuming that $(\e_k)$ is a subsequence along which both Theorem~\ref{thm:gamma_limit} and~\cite[Theorem 3.1]{BMZ} hold true, we can readily deduce that
	\begin{equation}\label{est:lb:f-g-tilde}
	{f}^\ell(\xi)\geq  \hat{f}^\ell(\xi)\quad\text{and}\quad {g}^\ell(\zeta,\nu)\geq\hat{g}^\ell(\nu)\,,
	\end{equation}
	for every $\xi\in\R^{m\times n}$, every $(\zeta,\nu)\in \R_0^m\times\S^{n-1}$, and every $\ell \in [0,+\infty]$. 
	
	In Proposition~\ref{prop:volume_density} below we show that ${f}^\ell$ and $\hat{f}^\ell$ coincide and that for every $\ell\in [0,+\infty]$ they are actually equal to $f_\hom$ as in \eqref{f_hom}. Furthermore, in Propositions~\ref{prop:supercritical_case} and~\ref{prop:critical_case} we prove that ${g}^\ell$ and $\hat{g}^\ell$ coincide for every $\ell\in [0,+\infty)$; therefore, in these regimes ${g}^\ell$ is independent of $\zeta$.
	Moreover, for $\ell=0$ there holds ${g}^0=\hat{g}^0 =g^0_\hom$, with $g^0_\hom$ given by \eqref{g_hom_supercritical}; while for $\ell \in (0,+\infty)$ there holds ${g}^\ell=\hat{g}^\ell =g^\ell_\hom$,  with $g_\hom^\ell$ as in \eqref{g_hom_critical}.

	Eventually, in Proposition \ref{prop:subcritical:case} we deal with the case $\ell=+\infty$ and prove that ${g}^\ell$ coincides $g_\hom^\infty$ as in \eqref{g_hom_subcritical}. This ensures, in particular, that ${g}^\ell$ is independent of $\zeta$ also when $\ell=+\infty$. We notice, however, that in this regime it is not clear whether ${g}^\ell$ and $\hat{g}^\ell$ coincide.  	
	
	\subsection{Characterisation of the volume integrand}\label{sec:volume}
	In this subsection we characterise the volume integrand ${f}^\ell$ in~\eqref{tilde-F}. Namely, we prove the following result. 
\begin{proposition}\label{prop:volume_density}
	Let $f_{\rm hom}$ and $\hat f^\ell$ be as in \eqref{f_hom} and \eqref{eq:densities-vol}, respectively. 
	Let moreover $f^\ell$ be the volume integrand in \eqref{tilde-F}. Then for every $\xi\in\R^{m\times n}$ and every $\ell\in[0,+\infty]$ there holds
\begin{equation}\label{eq:f-tilde-hat-hom}
f^\ell(\xi)=\hat f^\ell(\xi)=f_{\rm hom}(\xi).
\end{equation}
\end{proposition}
\begin{proof}
	To not to overburden notation the $\Gamma$-converging subsequence provided by Theorem~\ref{thm:gamma_limit} is still denoted by $(\F_\e)$. 
	
	We establish the two equalities in~\eqref{eq:f-tilde-hat-hom} separately.
	
	\medskip
	\step 1 $\hat f^\ell(\xi)=f_{\rm hom }(\xi)$, for every $\xi\in\R^{m\times n}$ and every $\ell\in [0,+\infty]$.
	
	\smallskip
	Let $\e>0$, $\rho>2\e$, and $u\in W^{1,2}(Q_\rho;\R^m)$ be fixed and define $u_\e\in W^{1,2}(Q_{\rho/\delta_\e};\R^m)$ by setting $u_\e(x):=\frac{1}{\delta_{\e}}u(\delta_{\e} x)$, for every $x\in Q_{\rho/\delta_\e}$. Then clearly $u=u_\xi$ near $\partial Q_\rho$ if and only if $u_\e=u_\xi$ near  $\partial Q_{\rho/\delta_{\e}}$. Moreover, setting $r_\e\defas\rho/\delta_\e$ a change of variables gives
	\begin{equation*}
		\frac{1}{\rho^n}\int_{Q_\rho}f\Big(\frac{x}{\delta_\e},\nabla u\Big)\dx=
		\frac{\delta_\e^n}{\rho^n}\int_{Q_{\rho/\delta_\e}}f\left(x,\nabla u_\e\right)dx=\frac{1}{r_\e^n}\int_{Q_{r_\e}}f\left(x,\nabla u_\e\right)dx\,.
	\end{equation*}
	Passing to the infimum in $u$ and appealing to Remark~\ref{rem:prop-f} we thus deduce that
	\begin{equation*}
		\lim_{\e\to 0}\frac{1}{\rho^n}\m_{\e}^b(u_\xi,Q_\rho)=	f_{\rm hom}(\xi),
	\end{equation*}
	where the limit above exists independently of $\rho>0$ and of  the $\Gamma$-converging subsequence. Then \eqref{eq:densities-vol} yields the claim.
	    
	    \medskip
		\step 2 $f^\ell(\xi)=f_{\rm hom }(\xi)$, for every $\xi\in\R^{m\times n}$ and every $\ell\in [0,+\infty]$.
		
		\smallskip
	By combining~\eqref{est:lb:f-g-tilde} and Step 1 we immediately deduce that
	\begin{equation}\label{lb:f-tilde}
		{f}^\ell(\xi)\geq f_\hom(\xi)\,,
	\end{equation}
	for every $\xi\in\R^{m\times n}$ and every $\ell\in [0,+\infty]$. 
	
	The proof of the opposite inequality follows by constructing a recovery sequence similarly as in the case of classical homogenisation in Sobolev spaces (see~\cite[Lemma 2.1(a)]{Mueller87}). For the readers' convenience we repeat this construction in our setting. Let $\sigma>0$ be fixed, thanks to~\eqref{f-hom-equiv} we can find $r\in\N^*$ and $u\in W^{1,2}(Q_r;\R^m)$ with $u=0$ near $\partial Q_r$ such that
	\begin{equation}\label{f-hom-almost-optimal}
		\frac{1}{r^n}\int_{Q_r}f(x,\nabla u+\xi)\dx\leq f_\hom(\xi)+\sigma\,.
	\end{equation} 
	We then extend $u$ $r$-periodically to $\R^n$ and define $(u_\e)\subset W^{1,2}_\loc(\R^n;\R^m)$ as 
	\[
	u_\e(x)\defas u_\xi(x)+\delta_\e u\Big(\frac{x}{\delta_\e}\Big).
	\]
	Clearly $({u}_\e)\subset W^{1,2}_\loc(\R^n;\R^m)$; moreover, by definition of $u_\e$ we have that $u_\e\wto u_\xi$ weakly in $W_\loc^{1,2}(\R^n;\R^m)$ and $u_\e\to u_\xi$ in $L_\loc^{2}(\R^n;\R^m)$, therefore $u_{\e}\to u_\xi$ in $L^0(\R^n;\R^m)$. Hence Theorem~\ref{thm:gamma_limit} implies that
	\begin{equation}\label{ub:tilde-f1}
	\begin{split}
		{f}^\ell(\xi)=\F^\ell(u_\xi,1,Q) &\leq\liminf_{\e\to 0}\F_{\e}({u}_{\e},1,Q)=\liminf_{\e\to 0}\int_{Q}(1+\eta_{\e})f\Big(\frac{x}{\delta_{\e}},\nabla{u}_{\e}\Big)\dx\,.
		\end{split}
	\end{equation}
	Set 
	\[
	f\Big(\frac{x}{\delta_{\e}},\nabla {u}_{\e}\Big)=g\Big(\frac{x}{\delta_{\e}}\Big)\; \text{ with }\; g(y)\defas f(y,\xi+\nabla u(y)).
	\] 
	Since $g$ is $r$-periodic, by the Riemann-Lebesgue Theorem we get that 
	\[
	g\Big(\frac{\cdot}{\delta_{\e}}\Big)\wto \frac{1}{r^n}\int_{Q_r}g(y)\dy \; \text{ weakly in $L^1(Q)$}, 
	\]
	so that in particular
	\begin{equation*}
		\lim_{\e\to 0}\int_{Q}f\Big(\frac{x}{\delta_{\e}},\nabla{u}_{\e}\Big)\dx=\frac{1}{r^n}\int_{Q_r}f(y,\xi+\nabla u)\dy\,.
	\end{equation*}
	Thus, since $\nabla u_{\e}$ is uniformly bounded in $L^2(Q;\R^{m\times n})$, using~\ref{hyp:growth-f} and combining~\eqref{f-hom-almost-optimal} and~\eqref{ub:tilde-f1} we obtain
	\begin{equation}\label{ub:tilde-f2}
		{f}^\ell(\xi)\leq f_\hom(\xi)+\sigma\,.
	\end{equation} 	
	Eventually, we conclude gathering~\eqref{lb:f-tilde} and \eqref{ub:tilde-f2}, and by the arbitrariness of $\sigma>0$. 
\end{proof}
		\section{Oscillations on a larger scale than the singular perturbation}\label{sec:larger}
	In this section we characterise $\hat g^\ell$ and $g^\ell$ in the regime $\ell=0$; the latter corresponds to the case where the scale of the oscillations $\delta_\e$ is much larger than the scale of the singular perturbation $\e$.
	\begin{proposition}\label{prop:supercritical_case}
	Assume that $\ell=0$. Let $g^0_\hom$ and $\hat g^0$ be as in~\eqref{g_hom_supercritical} and~\eqref{eq:densities-sur}, respectively. Let $g^0$ be the surface integrand in~\eqref{tilde-F}.
	Then for every $(\zeta,\nu)\in\R_0^m\times\S^{n-1}$ there holds
	\begin{equation*}
	{g}^0(\zeta,\nu)=\hat g^0(\nu)=g_\hom^0(\nu)\,.
	\end{equation*} 
	\end{proposition}
	\begin{proof}
	Not to overburden notation we still denote by $(\F_\e)$ the $\Gamma$-converging subsequence given by Theorem~\ref{thm:gamma_limit}.
	
	We introduce the function $\Psi:\R^n\times\R^n\to[0,+\infty)$ given by
		\begin{equation*}
			\Psi(x,w)\defas\sqrt{h(x,w)}
		\end{equation*}
	and observe that by~\ref{hyp:growth-h} we have
	\begin{equation}\label{eq:growth_psi}
	\sqrt{c_3}|w|\le \Psi(x,w)\le\sqrt{c_4}|w|\,,
	\end{equation}
	for every $x,w\in\R^n$. Moreover, by \ref{hyp:hom-h} $\Psi$ is positively $1$-homogeneous and symmetric in $w$. For $\nu\in\S^{n-1}$ and $r>0$ it is also convenient to introduce the following notation
	\begin{equation}\label{def:mpc}
	\m^{\rm pc}(u^\nu, Q_r^\nu)\defas\inf\Biggl\{\int_{S_u\cap Q^\nu_r}\Psi^{**}(x,\nu_u)\dHn\colon u\in BV(Q_r^\nu;\{0,e_1\})\,, u=u^\nu\ \text{near}\ \partial Q^\nu_r\Biggr\}\,,
	\end{equation}	 
	where $\Psi^{**}$ denotes the convex envelope of $\Psi$ in the second variable. In view of Remark~\ref{rem:prop-g} \eqref{rem:g-0} $g_\hom^0$ can be rewritten as
	\begin{equation}\label{eq:ghom-mpc}
	g_\hom^0(\nu)=\lim_{r\to+\infty}\frac{2}{r^{n-1}}\m^{\rm pc}(u^\nu, Q_r^\nu)\,.
	\end{equation}
	 By \eqref{est:lb:f-g-tilde}, it suffices to show that
	\begin{equation}\label{est:ghom-supercrit}
	\hat{g}^0(\nu)\geq g_\hom^0(\nu)\geq {g}^0(\zeta,\nu),
	\end{equation}
	for every $(\zeta,\nu)\in\R^m_0\times\S^{n-1}$.
	The proof of~\eqref{est:ghom-supercrit} will be carried out in two separate steps.

		\medskip
		\step 1 $\hat g^0(\nu)\ge g^0_{\rm hom}(\nu)$, for every $\nu\in\S^{n-1}$. 
		
		\smallskip
 
	Let $\nu\in\S^{n-1}$, $\e>0$, $\rho>2\e$ and $v\in\Adm(\bar{u}_\e^\nu,Q_\rho^\nu)$ be arbitrary. Then there exists $u\in W^{1,2}(Q_\rho^\nu;\R^m)$ such that
	\begin{equation}\label{cond:uv-supercrit}
	v\nabla u=0\ \text{a.e. in}\ Q_\rho^\nu\quad\text{and}\quad (u,v)=(\bar u^\nu_{\e},\bar v^\nu_{\e})\ \text{near}\ \partial Q^\nu_\rho\,.
	\end{equation}
	Starting from the pair $(u,v)$ we now construct suitable competitors for the minimisation problem defining $\m^{\rm pc}$ in~\eqref{def:mpc}.   
		To this end, we define the increasing function $\Phi\colon[0,1]\to[0,1/2]$ as
		\begin{equation*}
			\Phi(t)\defas\int_0^t(1-z)\dz=\frac12-\frac{(1-t)^2}{2}\,.
		\end{equation*}
		Then, from the Young Inequality together with the homogeneity of $\Psi$ we deduce
		\begin{align}\nonumber
\F_\e^s(v,Q_\rho^\nu)= \int_{Q_\rho^\nu}\bigg(\frac{(1-v)^2}{\e}+\e h\Big(\frac{x}{\delta_\e},\nabla v\Big)\bigg)\dx &\ge 2\int_{Q_\rho^\nu}\Psi\left(\frac{x}{\delta_\e},(1-v)\nabla v\right)\dx\\ \nonumber
& \ge 2\int_{Q_\rho^\nu}\Psi^{**}\left(\frac{x}{\delta_\e},(1-v)\nabla v\right)\dx
\\ \label{est:Young}
& =2 \int_{Q_\rho^\nu}\Psi^{**}\left(\frac{x}{\delta_\e},\nabla \Phi(v)\right)\dx\,.
\end{align}
		For $s\in[0,1/2)$ we define the sets
		\begin{equation*}
		E^s\defas\{x\in Q_\rho^\nu\colon\Phi(v(x))>s\}\,,
		\end{equation*}
		which have finite perimeter for $\L^1$-a.e. $s\in (0,1/2)$. In view of~\eqref{est:Young}, by a generalised Coarea Formula (see e.g. \cite[Lemma 2.4]{DM}) and the Mean Value Theorem, we find $t\in (0,1/2)$ such that
		\begin{equation}\label{eq:coarea}
			\begin{split}
			\F_\e^s(v,Q_\rho^\nu)
				&\geq 2\int_{Q_\rho^\nu}\Psi^{**}\Big(\frac{x}{\delta_\e},\nabla \Phi(v)\Big)\dx
				=2\int_0^{1/2}\bigg(\int_{\partial^* E^s}\Psi^{**} \Big(\frac{x}{\delta_\e},\nu_{E^s}\Big)\dHn\bigg)\ds\\
				&\geq\int_{\partial^*E^t}\Psi^{**}\Big(\frac{x}{\delta_\e},\nu_{E^t}\Big)\dHn\,,
			\end{split}
		\end{equation}
		where $\partial^*E^s$ and $\nu_{E^s}$ denote the reduced boundary of $E^s$ and the measure theoretic inner normal to $E^s$, respectively.
		
		A direct computation shows that $v(x)>1-\sqrt{1-2t}>0$ for $x\in E^t$. As a consequence, from~\eqref{cond:u-variable} in Remark~\ref{rem:f-g-tilde} we obtain
		\begin{equation}\label{cond:u-Es}
		u(x)\in\{0,e_1\}\ \text{for a.e.}\ x\in E^t\,.
		\end{equation}
		Since in addition $E^t$ has finite perimeter in $Q_\rho^\nu$, the two functions
		\begin{equation*}
		u^0\defas u\chi_{E^t}\,,\qquad u^1\defas u\chi_{E^t}+e_1(1-\chi_{E^t})
		\end{equation*}
		belong to $BV(Q_\rho^\nu;\{0,e_1\})$. Moreover, up to an $\H^{n-1}$-negligible set, $\partial^*E^t$ is the disjoint union of $S_{u^0}$ and $S_{u^1}$. Indeed, from~\eqref{cond:u-Es} we readily deduce that $\partial^*E^t\sm S_u$ is the disjoint union $J_{u^0}$ and $J_{u^1}$, where $J_{u^0}$ and $J_{u^1}$ are the set of approximate jump points of $u^0$ and $u^1$. Since $u\in W^{1,2}(Q_\rho^\nu;\R^m)$, we have $\H^{n-1}(S_u)=0$, while $\H^{n-1}(S_{u^0}\sm J_{u^0})=\H^{n-1}(S_{u^1}\sm J_{u^1})=0$ by the properties of $BV$-functions. Hence, the claim follows. Since moreover $\nu_{u^0}=\pm\nu_{E^t}=\nu_{u^1}$ $\H^{n-1}$-a.e. on $\partial^* E^t$, from~\eqref{eq:coarea} together with the symmetry of $\Psi$ we deduce that
		\begin{equation}\label{est:u0-u1}
		\F_\e^s(v,Q_\rho^\nu) \geq\int_{S_u^0\cap Q_\rho^\nu}\Psi^{**}\Big(\frac{x}{\delta_\e},\nu_{u^0}\Big)\dHn +\int_{S_u^1\cap Q_{\rho}^\nu}\Psi^{**}\Big(\frac{x}{\delta_\e},\nu_{u^1}\Big)\dHn
		\end{equation}
		We extend $u^0, u^1$ by $u^\nu$ to $Q_{(1+\delta_\e)\rho}^\nu$ without renaming them.  In this way, thanks to~\eqref{cond:uv-supercrit} we have 
		\begin{equation}\label{inclusion:Suk}
		S_{u^k}\cap\big(Q_{(1+\delta_\e)\rho}^\nu\sm Q_\rho^\nu\big)\subset\Big(\Pi^\nu\cap\big(Q_{(1+\delta_\e)\rho}^\nu\sm Q_\rho^\nu\big)\cup\big(\partial Q_\rho^\nu\cap\{|x\cdot\nu|\leq \e\}\big)\Big)\quad\text{for}\ k=0,1\,.
		\end{equation}
		Finally, for $r_\e\defas(1+\delta_\e)\rho/\delta_\e$ and $k=0,1$ let $u_\e^k\in BV(Q_{r_\e}^\nu;\{0,e_1\})$ be given by $u_\e^k(x)\defas u^k(\delta_\e x)$. Then $u_\e^k=u^\nu$ near $\partial Q_{r_\e}^\nu$ and $S_{u_\e^k}=\frac{1}{\delta_\e}S_{u^k}$. Thus, by the change of variables $y=x/\delta_\e$ we obtain
		\begin{equation*}
		\frac{1}{\rho^{n-1}}\int_{S_u^k\cap Q_{(1+\delta_\e)\rho}^\nu}\hspace*{-2em}\Psi^{**}\Big(\frac{x}{\delta_\e},\nu_{u^k}\Big)\dHn\geq\frac{1}{r_\e^{n-1}}\int_{S_{u_\e}^k\cap Q_{r_\e}^\nu}\hspace*{-1em}\Psi^{**}(x,\nu_{u_\e^k})\dHn\geq \frac{1}{r_\e^{n-1}}\m^{\rm pc}(u^\nu,Q_{r_\e}^\nu)\,,
		\end{equation*}
		which together with~\eqref{est:u0-u1} and~\eqref{inclusion:Suk} gives
		\begin{equation}\label{est:Fs-mpc}
		\frac{1}{\rho^{n-1}}\F_\e^s(v,Q_\rho^\nu)\geq \frac{2}{r_\e^{n-1}}\m^{\rm pc}(u^\nu,Q_{r_\e}^\nu)-2\sqrt{c_3}\bigg((1+\delta_\e)^{n-1}-1+2(n-1)\frac{\e}{\rho}\bigg)\,.
		\end{equation}
		Since $v\in\Adm(\bar{u}_\e^\nu,Q_\rho^\nu)$ was arbitrarily chosen, we can pass to the infimum on the left-hand side of~\eqref{est:Fs-mpc} and let $\e\to 0$ to deduce that
		\begin{equation*}
		\frac{1}{\rho^{n-1}}\liminf_{\e\to 0}\m_\e^s(\bar{u}_\e^\nu,Q_\rho^\nu)\geq\lim_{\e\to 0}\frac{2}{r_\e^{n-1}}\m^{\rm pc}(u^\nu,Q_{r_\e}^\nu)=g_\hom^0(\nu)\,,
		\end{equation*}
		where the last equality follows from~\eqref{eq:ghom-mpc}. By the very definition of $\hat{g}^0(\nu)$ in~\eqref{eq:densities-sur} we conclude the proof of Step 1 by letting $\rho\to 0$.

		\medskip
		\step 2 $g^0(\zeta,\nu)\le g^0_{\rm hom}(\nu)$, for every $(\zeta,\nu)\in\R_0^m\times\S^{n-1}$.
		
		\smallskip	
		Let $(\zeta,\nu)\in\R_0^m\times\S^{n-1}$; by Theorem \ref{thm:gamma_limit} we have that 
		\begin{equation*}
\tilde g^0(\zeta,\nu)=\F^0(u_\zeta^\nu,1,Q^\nu)\le\liminf_{\e\to0}\F_\e(u_\e,v_\e,Q^\nu)\,,
		\end{equation*}
		for any sequence $(u_\e,v_\e)$ converging to $(u_\zeta^\nu,1)$ in $L^0(\R^n;\R^m)\times L^0(\R^n)$. Then to conclude it is sufficient to construct a sequence $(\bar u_\e,\bar v_\e)$ converging to $(u_\zeta^\nu,1)$ in $L^0(\R^n;\R^m)\times L^0(\R^n)$ and such that 
		\begin{equation}\label{claim}
\limsup_{\e\to0}\F_\e(\bar u_\e,\bar v_\e,Q^\nu)\le g^0_{\rm hom}(\nu)\,.
		\end{equation}
		Moreover, since $g_\hom^0$ and $g^0(\zeta,\cdot)$ are continuous (\cf Remark~\ref{rem:prop-g}~\eqref{rem:g-0} and Remark~\ref{rem:prop-g-tilde}, respectively), it suffices to consider $\nu\in\S^{n-1}\cap\Q^n$, then the general case follows by density. 
		
		Let $\nu \in\S^{n-1}\cap\Q^n$ and let $R_\nu\in \Q^{n\times n}$ be an orthogonal matrix as in~\ref{Rn} so that $R_\nu e_n=\nu$.
		Then, for $\eta>0$ fixed we find $r\in \N^*$ with $rR_\nu\in\Z^{n\times n}$ and $u\in BV(Q^\nu_r;\{0,e_1\})$ with $u=u^\nu$ on $\partial Q^\nu_r$ satisfying
		\begin{equation}\label{eq:upp_bound}
			\frac{2}{r^{n-1}}\int_{{Q^\nu_r}\cap S_u}\Psi(x,\nu_u)\dHn\le g^0_{\rm hom}(\nu)+\eta\,.
		\end{equation}
	By Reshetnyak's continuity Theorem~\cite[Theorem 2.39]{AFP}, the continuity of $\Psi$, and \cite[Theorem 3.42]{AFP} we can assume without loss of generality that $S_u$ is of class $C^2$. Moreover, since the approximation of the set $\{u=e_1\}$ with smooth sets is a local procedure (\cf \cite[Remark 3.43]{AFP}) and $u=u^\nu$ in a neighbourhood of $\partial Q_r^\nu$, the boundary conditions satisfied by $u$ are not affected by the smoothing procedure.
	We then extend $u$ to $\R^n$ (without renaming it) in a way so that it is $r$ periodic in the directions $R_\nu e_i$, $i=1,\ldots, n-1$ and $u=u^\nu$ in $\{x\in\R^n\colon|x\cdot\nu|>r/2\}$. 
	In this way we have that
	\begin{equation}\label{inclusion:Su}
	S_u\subset\{x\in\R^n\colon|x\cdot\nu|<r/2\}\,.
	\end{equation}
	Since $S_u$ is of class $C^2$, denoting by $\d_{S_u}(x)\defas\dist(x,S_u)$ the distance from $S_u$, for $\alpha>0$ suitably small there is a unique projection $\pi_\alpha\colon\{x \colon \d_{S_u}(x)<\alpha\}\to S_u$ of class $C^2$.
	 We then set
	 \begin{equation*}
	 	\overline\nu(x)\defas\begin{cases}
	 		\nu_u((\pi_\alpha(x))&\text{if}\ \d_{S_u}(x)<\alpha\,,\\
	 		\nu&\text{otherwise}\,.
	 	\end{cases}	
	 \end{equation*}
	 By Remark~\ref{rem:op} we can find $T_\eta>0$ and $v_\eta\in W^{1,2}_\loc(\R^n)$, $v_\eta$ Lipschitz continuous,  with
	 $v_\eta(0)=0$,  $v_\eta(t)=1$ for $t\geq T_\eta$, $0\leq v_\eta\leq 1$, and
		\begin{equation}\label{eq:c_eta}
			C_\eta\defas\int_0^{+\infty}\left((1-v_\eta)^2+(v'_\eta)^2\right)\dt\leq 1+\eta\,.
		\end{equation}
Next we choose $\xi_\e\defas\sqrt{\e\eta_\e}$, so that $\xi_\e\ll\e$, and consider the pair $(\bar u_\e, \bar v_\e)$ given by  
\[
(\bar u_\e(x), \bar v_\e(x))\defas( u_\e(x/\delta_\e),  v_\e(x/\delta_\e))
\] 
where
		\begin{equation*}
			 u_\e(x)\defas\begin{cases}
				\displaystyle\left(1-\frac{\dist(x,\{u=e_1\})\delta_\e}{\xi_\e}\right)\zeta &\text{ if } \dist(x,\{u=e_1\})<\dfrac{\xi_\e}{\delta_\e}\,,\\
				0&\text{ otherwise}\,,
			\end{cases}
		\end{equation*} 
		and
		\begin{equation*}
			 v_\e(x)\defas\begin{cases}
				0&\text{ if } \d_{S_u}(x)\le\dfrac{\xi_\e}{\delta_\e}\,,\\
				\displaystyle v_\eta\left(\frac{\delta_\e}{\e}\frac{ \d_{S_u}(x)-\frac{\xi_\e}{\delta_\e}}{\Psi(x,\overline\nu(x))}\right)&\text{ if } \d_{S_u}(x)>\dfrac{\xi_\e}{\delta_\e}\,,
			\end{cases}
		\end{equation*}
	for every $x\in\R^n$. Clearly, $\bar u_\e\in W^{1,2}_\loc(\R^n;\R^m)$. Moreover, by \ref{hyp:growth-h} and \eqref{eq:growth_psi}
	\begin{equation}\label{est:dSu}
		(u_\e(x), v_\e(x))=(u_\zeta^\nu(x),1) \quad\text{ if }\quad \d_{S_u}\left(x\right)\ge\frac{\xi_\e}{\delta_\e}+ \frac{\e}{\delta_\e}\sqrt{c_4}T_\eta\,.
	\end{equation}
	As a consequence, since $\xi_\e\ll\e\ll\delta_\e$, for $\e$ sufficiently small we have that that $\d_{S_u}(x)<\alpha$ whenever $v_\e(x)\neq 1$. In particular, both the mappings $x\mapsto\d_{S_u}(x)$ and $x\mapsto\bar\nu(x)$ are Lipschitz continuous in the region where $v_\e\not\equiv 1$. Thus, in view of~\eqref{eq:growth_psi}, the regularity of $\Psi$ and $v_\eta$ ensure that $\bar{v}_\e\in W^{1,\infty}_\loc(\R^n)$.
Moreover, thanks to~\eqref{inclusion:Su} we have that $\d_{S_u}\big(\frac{x}{\delta_\e}\big)\geq\frac{\xi_\e}{\delta_\e}+\frac{\e}{\delta_\e}\sqrt{c_4}T_\eta$ if $|x\cdot\nu|\geq r$, for $\e$ sufficiently small. Thus,~\eqref{est:dSu} implies that
	\begin{equation*}
		(\bar u_\e(x),\bar v_\e(x))=\big(u_\zeta^\nu\big(\tfrac{x}{\delta_\e}\big),1\big)=(u_\zeta^\nu(x),1)\quad\text{ if }\quad |x\cdot\nu|>r\delta_\e\,,
	\end{equation*}
and thus $(\bar u_\e,\bar v_\e)$ converges to $(u_\zeta^\nu,1)$ in 
	$L^0(\R^n;\R^m)\times L^0(\R^n)$.

	Now it remains to estimate $\F_\e$ along the sequence $(\bar u_\e,\bar v_\e)$. To this end we start observing that
		\begin{equation*}
			\bar v_\e\nabla \bar u_\e=0\quad\text{ a.e. in }\in \R^n\,,
		\end{equation*}
	hence from~\ref{hyp:growth-f} we deduce that
	\begin{equation}\label{Fs(barv)}
		\F_\e(\bar u_\e,\bar v_\e,Q^\nu)\leq c_2\eta_\e\int_{Q^\nu}|\nabla \bar{u}_\e|^2\dx+\F^{s}_\e(\bar v_\e,Q^\nu)\,.
	\end{equation}
In order to estimate the right-hand side of \eqref{Fs(barv)} it is convenient to define the sets
		\begin{equation*}
			A_\e\defas\left\{x\in Q^\nu\colon\d_{S_u}\left(\frac{x}{\delta_\e}\right)\le\frac{\xi_\e}{\delta_\e}\right\}\,,
		\end{equation*}
	\begin{equation*}
	B_\e\defas\left\{x\in Q^\nu\colon\frac{\xi_\e}{\delta_\e}<\d_{S_u}\left(\frac{x}{\delta_\e}\right)<\frac{\xi_\e}{\delta_\e}+\frac{\e}{\delta_\e}\sqrt{c_4}T_\eta\right\}\,.
	\end{equation*}
	By definition of $\bar{u}_\e$ we have
	\begin{equation}\label{est:bulk-supercrit}
	\eta_\e\int_{Q^\nu}|\nabla \bar{u}_\e|^2\dx\leq\frac{C\eta_\e}{\xi_\e^2}\L^n(A_\e)=C\frac{\L^n(A_\e)}{\e}\,.
	\end{equation}
		Moreover,~\ref{hyp:growth-h} implies that $h\big(\frac{x}{\delta_\e},\nabla\bar{v}_\e(x)\big)=0$, if $x\in A_\e$ and $\F_\e^s(\bar{v}_\e, Q^\nu\sm(\overline A_\e \cup \overline{B}_\e))=0$. Therefore, we infer 
		\begin{equation}\label{eq:surf_term}
			\begin{split}
				\F^{s}_\e(\bar v_\e,Q^\nu)
				=\frac{\L^n(A_\e)}{\e}
				+	\F^{s}_\e(\bar v_\e,B_\e)
				\,.
			\end{split}
		\end{equation}
We now show that $\L^n(A_\e)/\e$ vanishes as $\e$ tends to zero. Since 
		\begin{equation*}
		A_\e=\delta_\e\bigg\{x\in Q_{1/\delta_\e}^\nu\colon\d_{S_u}\left(x\right)\le\frac{\xi_\e}{\delta_\e}\bigg\}
	\end{equation*}
by~\eqref{inclusion:Su} and the periodicity of $u$, if we cover $\{x\in Q_{1/\delta_\e}^\nu\colon\d_{S_u}\left(x\right)\le\frac{\xi_\e}{\delta_\e}\}$ with $(\lfloor1/(\delta_\e r)\rfloor+1)^{n-1}$
copies of $\{x\in Q^\nu_r\colon\d_{S_u}(x)\le\frac{\xi_\e}{\delta_\e}\}$ we get
\begin{equation*}
\L^n(A_\e)\le  \delta_\e\frac{(1+\delta_\e r)^{n-1}}{r^{n-1}}\L^n\bigg(\bigg\{x\in Q^\nu_r\colon\d_{S_u}(x)\le\frac{\xi_\e}{\delta_\e}\bigg\}\bigg)\,.
\end{equation*}
Since $S_u\cap Q^\nu_r$ is of class $C^2$, the $(n-1)$-dimensional Minkowski content of $S_u\cap Q^\nu_r$ coincides with $\mathcal{H}^{n-1}(S_u\cap Q^\nu_r)$, therefore
\begin{equation*}
\L^n\bigg(\bigg\{x\in Q^\nu_r\colon\d_{S_u}(x)\le\frac{\xi_\e}{\delta_\e}\bigg\}\bigg)= \mathcal{H}^{n-1}(S_u\cap Q^\nu_r)\frac{\xi_\e}{\delta_\e}+ O\left(\frac{\xi_\e}{\delta_\e}\right)\,;
\end{equation*}
then, since $\xi_\e\ll\e$, we have that in particular
\begin{equation}\label{eq:surf_term_1}
\frac{\L^n(A_\e)}{\e}\le\frac{\mathcal{H}^{n-1}(S_u\cap Q^\nu_{r})\xi_\e}{\e}\frac{(1+\delta_\e r)^{n-1}}{r^{n-1}}+o(1)=o(1)\,,
\end{equation}
as $\e\to0$. Hence, gathering~\eqref{Fs(barv)}-\eqref{eq:surf_term_1} implies
\begin{equation}\label{est:Fk-surf-subcrit}
\F_\e(\bar{u}_\e,\bar{v}_\e,Q^\nu)\leq\F_\e^s(\bar v_\e,B_\e)+o(1)\,,
\end{equation}
so that it only remains to estimate $\F_\e^s(\bar v_\e,B_\e)$. To do so set
		\begin{equation*}
g_\e(x)\defas\frac{\delta_\e}{\e}\frac{\d_{S_u}\Big(\frac{x}{\delta_\e}\Big)-\frac{\xi_\e}{\delta_\e}}{\Psi\left(\frac{x}{\delta_\e},\bar\nu\left(\frac{x}{\delta_\e}\right)\right)}\,,
		\end{equation*}
		thus $\bar v_\e(x)=v_\eta(g_\e(x))$. Therefore~\ref{hyp:hom-h} implies that
\begin{equation}\label{eq:Fks-Bk}
\F_\e^s(\bar{v}_\e,B_\e)=\int_{B_\e}\frac{(1-v_\eta(g_\e(x)))^2}{\e}+ \e (v_\eta'(g_\e(x))^2h\left(\frac{x}{\delta_\e},\nabla g_\e(x)\right)\dx\,.
\end{equation}		 
Moreover, by using that $\nabla\d_{S_u}(x/\delta_\e)=\delta_\e^{-1}\overline\nu(x/\delta_\e)$ we have 
\begin{equation}\label{eq:grad-gk}
\begin{split}
\nabla g_\e(x)
=	\frac{G_\e(x)}{\e\Psi\Big(\frac{x}{\delta_\e},\bar\nu\Big(\frac{x}{\delta_\e}\Big)\Big)}\,,
\end{split}
\end{equation}
where
\begin{equation*}
G_\e(x)\defas 
\Bigg(\bar\nu\Big(\frac{x}{\delta_\e}\Big) -\frac{\Big(\d_{S_u}\Big(\frac{x}{\delta_\e}\Big)-\frac{\xi_\e}{\delta_\e}\Big)\Big(\nabla_y\Psi\Big(\frac{x}{\delta_\e},\bar\nu\Big(\frac{x}{\delta_\e}\Big)\Big)+\nabla\bar\nu\Big(\frac{x}{\delta_\e}\Big)\nabla_w\Psi\Big(\frac{x}{\delta_\e},\bar\nu\Big(\frac{x}{\delta_\e}\Big)\Big)\Big)}{\Psi\left(\frac{x}{\delta_\e},\bar\nu\left(\frac{x}{\delta_\e}\right)\right)}
\Bigg)\,.
\end{equation*}	
Notice that for $x\in B_\e$ we have that $\nabla\overline\nu(x/\delta_\e)=\nabla(\nu_u(\pi_{\alpha}(x/\delta_\e)))$ for $\e$ small enough, moreover~\eqref{eq:growth_psi},\ref{hyp:cont-h}, and~\ref{hyp:Lip-h} yield
	\begin{equation}\label{eq:estimates_psi}
	\left|	\Psi\left(\frac{x}{\delta_\e},\bar\nu\left(\frac{x}{\delta_\e}\right)\right)\right|^{-1}\le c\,,\quad
	\left|\nabla_y\Psi\Big(\frac{x}{\delta_\e},\bar\nu\Big(\frac{x}{\delta_\e}\Big)\Big)\right|\le c\,,\quad
	\left|\nabla_w\Psi\Big(\frac{x}{\delta_\e},\bar\nu\Big(\frac{x}{\delta_\e}\Big)\Big)	\right|\le c\,,
\end{equation}
for some $c=c(n,L_2,L_3,c_3)>0$.
Hence \eqref{eq:estimates_psi} in particular implies that 
\[
G_\e(x)=\bar\nu\Big(\frac{x}{\delta_\e}\Big)+O\Big(\frac{\e}{\delta_\e}\Big)\quad \text{for $x\in B_\e$}\,,
\]
which together with~\eqref{eq:grad-gk}, \ref{hyp:hom-h}, and~\ref{hyp:cont-h} give
\begin{equation*}
			h\Big(\frac{x}{\delta_\e},\nabla g_\e(x)\Big)=\frac{h\big(\frac{x}{\delta_\e},G_\e(x)\big)}{\e^2\Psi^2\big(\frac{x}{\delta_\e},\bar{\nu}(\frac{x}{\delta_\e})\big)}=\frac{h\big(\frac{x}{\delta_\e},\bar{\nu}(\frac{x}{\delta_\e})\big)+O\big(\frac{\e}{\delta_\e}\big)}{\e^2\Psi^2\big(\frac{x}{\delta_\e},\bar{\nu}(\frac{x}{\delta_\e})\big)}\,.
		\end{equation*}
Thus, using that $\Psi^2=h$, from~\eqref{eq:Fks-Bk} we deduce that
			\begin{equation}\label{eq:FB_1}
			\begin{split}
				\F_\e^{s}(\bar v_\e,B_\e)
				&\le\left(1+O\left(\frac{\e}{\delta_\e}\right)\right) 		\int_{B_\e}\left(\frac{(1-v_\eta(g_\e(x)))^2}{\e}+ \frac{( v'_\eta(g_\e(x)))^2}{\e} \right)
				\dx\,.
			\end{split}
		\end{equation}
		In order to estimate the right-hand side of~\eqref{eq:FB_1}, it is convenient to write $B_\e$ as the disjoint union of the sets $B_\e^+$ and $B_\e^-$ defined as $B_\e^+\defas B_\e\cap\big\{x\colon u(\frac{x}{\delta_\e})=e_1\big\}$, $B_\e^-\defas B_\e\cap\big\{x\colon u(\frac{x}{\delta_\e})=0\big\}$.
		On $B_\e^+$ we use the change of variables $x= \delta_\e(y+ t\nu_u(y))$ for $y\in S_u\cap Q^\nu_\rho$  and $t\defas\frac{\delta_\e}{\e}\big(\d_{S_u}(\frac{x}{\delta_\e})-\frac{\xi_\e}{\delta_\e}\big)$. Note that in this way we have $\overline\nu(x/\delta_\e)=\nu_u(y) $. Since also $|\delta_\e\nabla\d_{S_u}(x/\delta_\e)|=1$, using the Coarea Formula on the right hand-side of \eqref{eq:FB_1} restricted to $B_\e^+$, we get
			\begin{equation}\label{eq:FB_2}
			\begin{split}
				\F_\e^{s}(\bar v_\e,B_\e^+)
				\leq&\left(1+ O\left(\frac{\e}{\delta_\e}\right)\right)\delta_\e^{n-1}\int_0^{\sqrt{c_4}T_\eta}\int_{{Q^\nu_{1/\delta_\e}}\cap S_u}\Biggl(1-v_\eta\Biggl(\frac{t}{\Psi\Big(y+\frac{\e}{\delta_\e}t\nu_u(y),\nu_u(y)\Big)}\Biggr)	\Biggr)^2
				\\
				&\qquad\qquad\qquad\qquad\qquad\qquad\qquad+\Biggl(v'_\eta\Biggl(\frac{t}{\Psi\Big(y+\frac{\e}{\delta_\e}t\nu_u(y),\nu_u(y)\Big)}\Biggr)\Biggr)^2\dHn(y)\dt\,.
			\end{split}
		\end{equation}
	Now let
		\begin{equation}\label{change_variable}
			s\defas\frac{t}{\Psi\Big(y+\frac{\e}{\delta_\e}t\nu_u(y),\nu_u(y)\Big)}\,,
		\end{equation}
	so that by \eqref{eq:estimates_psi} there holds
	\begin{equation*}
\frac{ds}{d t}=\frac{\Psi\Big(y+\frac{\e}{\delta_\e}t\nu_u(y),\nu_u(y)\Big)-t\Psi_x\Big(y+\frac{\e}{\delta_\e}t\nu_u(y),\nu_u(y)\Big)\frac{\e}{\delta_\e}\nu_u(y)}{\Psi^2\Big(y+\frac{\e}{\delta_\e}t\nu_u(y),\nu_u(y)\Big)}=\frac{1+O\left(\frac{\e}{\delta_\e}\right)}{\Psi\Big(y+\frac{\e}{\delta_\e}t\nu_u(y),\nu_u(y)\Big)}\,.
	\end{equation*}
Moreover, by \ref{hyp:Lip-h} we deduce that for every $t\in (0,\sqrt{c_4}T_\eta)$ we have
\begin{equation*}
\Psi\Big(y+\frac{\e}{\delta_\e}t\nu_u(y),\nu_u(y)\Big)=\Psi(y,\nu_u(y))+O\left(\frac{\e}{\delta_\e}\right)\,.
\end{equation*}
Hence using the change of variables in~\eqref{change_variable} and applying Fubini's Theorem in~\eqref{eq:FB_2} we infer
		\begin{equation}\label{eq:surf_term_3}
			\begin{split}
				\F_\e^{s}(\bar v_\e,B_\e^+)&
				\le \frac{1+ O\big(\frac{\e}{\delta_\e}\big)}{1+O\big(\frac{\e}{\delta_\e}\big)}
				\delta_\e^{n-1}C_\eta
				\int_{{Q^\nu_{1/\delta_\e}}\cap S_u}\bigg(
				\Psi(y,\nu_u(y))+O\Big(\frac{\e}{\delta_\e}\Big)\bigg)\dHn(y)
					\,,
			\end{split}
		\end{equation} 
	where $C_\eta$  is defined in \eqref{eq:c_eta}. 
	Moreover, a similar estimate can be obtained on $B_\e^-$ using the change of variables $x= \delta_\e(y- t\nu_u(y))$ for $y\in S_u\cap Q^\nu_\rho$.
	Thus, by the periodicity of $u$ and using~\eqref{inclusion:Su}, from \eqref{eq:surf_term_3} together with the periodicity of $\Psi$ and the fact that $rR_\nu\in\Z^{n\times n}$ we deduce 
			\begin{equation}\label{eq:surf_term_4}
		\begin{split}
			\F_\e^{s}(v_\e,B_\e)&
			\le \frac{2+ O\big(\frac{\e}{\delta_\e}\big)}{1+O\big(\frac{\e}{\delta_\e}\big)}
	\frac	{	\left(1+\delta_\e r\right)^{n-1}}{r^{n-1}}
			C_\eta
			\int_{{Q^\nu_{r}}\cap S_u}\bigg(
			\Psi(y,\nu_u(y))+O\Big(\frac{\e}{\delta_\e}\Big)\bigg)\dHn(y)
			\,.
		\end{split}
	\end{equation} 
	Thus, thanks to~\eqref{eq:c_eta}, \eqref{est:Fk-surf-subcrit}, and~\eqref{eq:surf_term_4} we obtain
	\begin{equation*}
\limsup_{\e\to0}\F_\e(\bar u_\e,\bar v_\e,Q^\nu)\le (1+\eta)(g^0_{\rm hom}(\nu)+\eta)\,.
	\end{equation*}
	 Eventually, \eqref{claim} follows by the arbitrariness of $\eta>0$, using a diagonal argument.
	\end{proof}
	\begin{remark}
	We notice that Step 1 in the proof of Proposition~\ref{prop:supercritical_case} could have been established also without using the asymptotic minimisation formula for $\hat{g}^0$ in~\eqref{eq:densities-sur} and instead using a similar argument as in~\cite[Theorem 17]{BCS} now appealing to the homogenisation result~\cite[Theorem 1]{Sol}. Moreover, Step 1 holds true also when $\delta_\e\ll\e$ and $\delta_\e\sim\e$. In particular, we have $g_\hom^0\leq g_\hom^\ell$ for any $\ell\in(0,+\infty]$.
	\end{remark}
	\section{Oscillations on the same scale as the singular perturbation}\label{sec:same}
	In this section we characterise $\hat g^\ell$ and $g^\ell$ in the regime $\ell \in (0,+\infty)$; the latter corresponds to the case where the scale of the oscillations $\delta_\e$ is comparable to the scale of the singular perturbation $\e$.
	\begin{proposition}\label{prop:critical_case}
	Assume that $\ell\in (0,+\infty)$. Let $g^\ell_\hom$ and $\hat g^\ell$ be as in~\eqref{g_hom_critical} and~\eqref{eq:densities-sur}, respectively. Let $g^\ell$ be the surface integrand in~\eqref{tilde-F}.
	Then for every $(\zeta,\nu)\in\R_0^m\times\S^{n-1}$ we have
	\begin{equation*}
		g^\ell(\zeta,\nu)=\hat g^\ell(\nu)=g_\hom^\ell(\nu)\,.
	\end{equation*} 
	\end{proposition}
	\begin{proof}
		Not to overburden notation, in all that follows $(\F_\e)$ denotes the $\Gamma$-converging subsequence given by Theorem~\ref{thm:gamma_limit}. 
		
		We recall that in view of Remark~\ref{rem:f-g-tilde} we have
		\begin{equation}\label{eq:g-hat-nu}
		\hat	g^\ell(\nu)= \limsup_{\rho\to0}\frac{1}{\rho^{n-1}}\limsup_{\e\to 0}\m^{s}_{\e}( \bar{u}_{\ell\delta_\e}^\nu,Q^\nu_\rho)= \limsup_{\rho\to0}\frac{1}{\rho^{n-1}}\liminf_{k\to\infty}\m^{s}_{\e}(\bar u^\nu_{\ell\delta_\e},Q^\nu_\rho)\,,
		\end{equation}
		for every $\nu\in\S^{n-1}$. It is also convenient to introduce the notation
		\begin{equation*}
		\m^{\ell,s}(\bar u^\nu,Q_r^\nu)\defas\inf\Bigg\{\int_{Q_r^\nu}\big((1-v)^2+h(\ell x,\nabla v)\big)\dx\colon v\in\Adm(\bar{u}^\nu,Q_r^\nu)\Bigg\}\,,
		\end{equation*}
		where $\Adm(\bar{u}^\nu,Q_r^\nu)$ is defined according to~\eqref{def_A}. Therefore $g_\hom^\ell$ can be rewritten as
		\begin{equation}\label{eq:g-hom-crit}
		g_\hom^\ell(\nu)=\lim_{r\to+\infty}\frac{1}{r^{n-1}}\m^{\ell,s}(\bar{u}^\nu,Q_r^\nu)\,.
		\end{equation}
		We notice that in view of~\eqref{est:lb:f-g-tilde}, to prove the claim it suffices to show that
		\begin{equation}\label{est:ghom-crit}
		\hat{g}^\ell(\nu)\geq g_\hom^\ell(\nu)\geq{g}^\ell(\zeta,\nu),\quad\forall\ (\zeta,\nu)\in\R^m_0\times\S^{n-1}\,.
		\end{equation}
		The proof of~\eqref{est:ghom-crit} will be split into two steps.
		
		\medskip
		\step 1 $\hat{g}^\ell(\nu)\geq g_\hom^\ell(\nu)$, for every $\nu\in\S^{n-1}$.
		
		\smallskip
		Let $\nu\in\S^{n-1}$, $\e>0$, $\rho>2\e$, and let $v\in\mathscr A(\bar u^\nu_{\ell\delta_\e},Q_\rho^\nu) $ be arbitrary. 
Then, there exists $u\in W^{1,2}(Q_\rho^\nu;\R^m)$ such that 
\begin{equation}\label{cond:v-crit}
v\nabla u=0\ \text{a.e. in}\ Q_\rho^\nu\quad\text{and}\quad (u,v)=(\bar u^\nu_{\ell\delta_\e},\bar v^\nu_{\ell\delta_\e})\ \text{near}\ \partial Q^\nu_\rho\,.
\end{equation} 
Set $r_\e\defas\frac{\rho}{\ell\delta_\e}$ and define $(u_\e,v_\e)\subset W^{1,2}(Q_{r_\e}^\nu;\R^m)\times W^{1,2}(Q_{r_\e}^\nu)$ by setting 
\[
(u_\e(x),v_\e(x))\defas(u(\ell\delta_\e x),v(\ell\delta_\e x)).
\] 
Then~\eqref{cond:v-crit} implies that 
\begin{equation*}
v_\e \nabla u_\e=0\ \text{a.e.\ in }\ Q^\nu_{r_\e}\ \text{ and }\ (u_\e,v_\e)=(\bar u^\nu,\bar v^\nu)\ \text{near}\ \partial Q^\nu_{r_\e}\,,
\end{equation*}
that is, $v_\e\in\Adm(\bar u^\nu,Q_{r_\e}^\nu)$. Thus, a change of variables gives 
	\begin{equation}\label{est:m-ell-s}
	\begin{split}
\frac{1}{\rho^{n-1}}\int_{Q_\rho^\nu}\bigg(\frac{(1-v)^2}{\e}+\e h\Big(\frac{x}{\delta_\e}\nabla v\Big)\bigg)\dx &=
\frac{(\ell\delta_\e)^{n-1}}{\rho^{n-1}}\int_{Q^\nu_{r_\e}}\left(\frac{\ell\delta_\e}{\e}(1-v_\e)^2+\frac{\eps}{\ell\delta_\e} h\left(\ell x,\nabla v_\e\right)\right)\dx\\
&\geq\frac{\gamma_\e}{r_\e^{n-1}} \m^{\ell,s}\big(\bar{u}^\nu,Q_{r_\e}^\nu\big)\,,
\end{split}
	\end{equation}
	where
	\begin{equation*}
	\gamma_\e\defas\min\left\{\frac{\ell\delta_\e}{\e},\frac{\e}{\ell\delta_\e}\right\}\ \to 1\quad\text{as }\; \e\to 0\,.
	\end{equation*}
	Hence, since $v\in\Adm(\bar{u}_{\ell\delta_\e}^\nu,Q_\rho^\nu)$ was arbitrarily chosen, we can pass to the infimum on the left-hand side of~\eqref{est:m-ell-s} and let $\e\to 0$ to deduce that
	\begin{equation*}
	\frac{1}{\rho^{n-1}}\liminf_{\e\to 0}\m_\e^s\big(\bar{u}_{\ell\delta_\e}^\nu, Q_\rho^\nu\big)\geq\lim_{\e\to 0}\frac{1}{r_\e^{n-1}}\m^{\ell,s}\big(\bar{u}^\nu, Q_{r_\e}^\nu\big)=g_\hom^\ell(\nu)\,,
	\end{equation*}
	where the last equality follows from~\eqref{eq:g-hom-crit}. In view of~\eqref{eq:g-hat-nu} we then conclude the proof of Step 1 by letting $\rho\to 0$.
	
	\medskip
	\step 2 ${g}^\ell(\zeta,\nu)\leq g_\hom^\ell(\nu)$, for every $(\zeta,\nu)\in\R_0^m\times\S^{n-1}$.
	
	\smallskip
	Let $(\zeta,\nu)\in\R^m_0\times\S^{n-1}$; by Theorem \ref{thm:gamma_limit} we have that 
		\begin{equation}\label{est:tilde-g-crit}
g^\ell(\zeta,\nu)=\F^\ell(u_\zeta^\nu,1,Q^\nu)\le\liminf_{\e\to0}\F_\e(u_\e,v_\e,Q^\nu)\,,
		\end{equation}
		for any sequence $(u_\e,v_\e)$ converging to $(u_\zeta^\nu,1)$ in $L^0(\R^n;\R^m)\times L^0(\R^n)$. 
		
		Let $\eta>0$ be fixed, in what follows we construct a sequence $(\bar u_\e,\bar v_\e)$ converging to $(u_\zeta^{\nu},1)$ in $L^0(\R^n;\R^m)\times L^0(\R^n)$ and such that 
		\begin{equation}\label{c:claim}
\limsup_{\e\to0}\F_\e(\bar u_\e,\bar v_\e,Q^\nu)\le g^\ell_{\rm hom}(\nu)+\eta\,.
		\end{equation}
		Then, we can conclude by combining~\eqref{est:tilde-g-crit} and~\eqref{c:claim} and by the arbitrariness of $\eta>0$.

We notice that since both $g_\hom^\ell$ and ${g}^\ell(\zeta,\cdot)$ are continuous (\cf Remark~\ref{rem:prop-g}~\eqref{rem:g-ell} and Remark~\ref{rem:prop-g-tilde}) we can prove the desired inequality only for $\nu\in\S^{n-1}\cap\Q^n$ and then conclude by density. Let $\nu \in\S^{n-1}\cap\Q^n$ and let $R_\nu\in \Q^{n\times n}$ be an orthogonal matrix as in~\ref{Rn} with $R_\nu e_n=\nu$, and let $m_\nu\in\N^*$ be such that $m_\nu R_\nu\in\Z^{n\times n}$. 
	
By~\eqref{eq:g-hom-crit} we find $r\in \frac{m_\nu}{\ell}\N^*$ and $v\in\Adm(\bar{u}^\nu, Q_{r}^\nu)$ such that
	\begin{equation}\label{almost-optimal-ghomell}
	\frac{1}{r^{n-1}}\int_{Q_{r}^\nu}\big((1-v)^2+h(\ell x,\nabla v)\big)\dx\leq g_\hom^\ell(\nu)+\eta\,.
	\end{equation}
	By definition of $\Adm(\bar{u}^\nu, Q_{r}^\nu)$, there exists $u\in W^{1,2}(Q_{r}^\nu;\R^m)$ such that
	\begin{equation}\label{cond:uv-ub-crit}
	v \nabla u=0\ \text{a.e.\ in }\ Q^\nu_{r}\ \text{ and }\ (u,v)=(\bar u^\nu,\bar v^\nu)\ \text{near}\ \partial Q^\nu_{r}\,.
	\end{equation}
	For $\lambda>0$ we introduce
	\begin{equation*}
	S_\lambda^\nu\defas\bigg\{x\in\R^n\colon|x\cdot\nu|<\frac{\lambda}{2}\bigg\}\,
	\end{equation*}
	the strip of width $\lambda$ around the hyperplane $\Pi^\nu$.
	We then extend $(u,v)$ $r$-periodically inside the strip $S_r^\nu$ by setting
	\begin{equation}\label{extension-uv}
	\big(u(x),v(x)\big)\defas \big(u(x-R_\nu(rz,0)),v(x-R_\nu(rz,0))\big)\; \text{ if }\; x\in Q_{r}^\nu(R_\nu(rz,0))\,,\ z\in\Z^{n-1}\,, 
	\end{equation}
	and we set $(u,v)\defas(u^\nu,1)$ in $\R^n\setminus S_r^\nu$. Then the second condition in~\eqref{cond:uv-ub-crit} ensures that $(u,v)\in W^{1,2}_\loc(\R^n;\R^m)\times W^{1,2}_\loc(\R^n)$. Eventually, we define $(\bar u_\e,\bar v_\e)\in W^{1,2}_\loc(\R^n;\R^m)\times W^{1,2}_\loc(\R^n)$ by setting 
	\begin{equation*}
	\bar u_\e(x)\defas \Big(u\Big(\frac{x}{\ell\delta_\e}\Big)\cdot e_1\Big)\zeta\quad\text{and}\quad \bar v_\e(x)\defas v\Big(\frac{x}{\ell\delta_\e}\Big)\,,
	\end{equation*}
	and we set $r_\e\defas\ell\delta_\e r\in \delta_\e m_\nu\N^*$.
	In this way, $\bar u_\e$ and $\bar v_\e$ are $r_\e$-periodic inside $S_{r_\e}^\nu$ and for $x\in\R^n\setminus S_{r_\e}^\nu$ we have $\bar u_\e(x)=\big(u^\nu\big(\frac{x}{\ell\delta_\e}\big)\cdot e_1\big)\zeta=u_\zeta^\nu(x)$ and $\bar v_\e(x)=1$. Thus, $(\bar u_\e,\bar v_\e)\to (u_\zeta^\nu,1)$ in $\LtL$. 
	To estimate $F_\e(\bar u_\e,\bar v_\e, Q^\nu)$ we note that $\bar v_\e\nabla \bar u_\e=0$ a.e.\! in $Q^\nu$ thanks to~\eqref{cond:uv-ub-crit}. Moreover, $\nabla \bar u_\e=0$ on $\R^n\setminus S_{r_\e}^\nu$ and $|\nabla \bar u_\e(x)|\leq\frac{1}{\ell\delta_\e}\big|\nabla \bar u\big(\frac{x}{\ell\delta_\e}\big)\big||\zeta|$ in $S_{r_\e}^\nu$. Thus, \ref{hyp:growth-f} together with a change of variables give
	\begin{equation}\label{est:bulk-crit-1}
	\int_{Q^\nu}(\bar v_\e^2+\eta_\e)f\Big(\frac{x}{\delta_\e},\nabla \bar u_\e\Big)\dx\leq c_2\eta_\e\int_{Q^\nu\cap S_{r_\e}^\nu}|\nabla \bar u_\e|^2\dx\leq c_2|\zeta|^2(\ell\delta_\e)^{n-2}\eta_\e\int_{Q_{1/(\ell\delta_\e)}^\nu\cap S_r^\nu}|\nabla u|^2\dx\,.
	\end{equation}
	By setting $J_\e^\nu\defas\{z\in \Z^{n-1}\colon Q_{r}^\nu(R_\nu(rz,0))\cap Q_{1/(\ell\delta_\e)}^\nu\neq \emptyset\}$ and using~\eqref{extension-uv} we obtain
	\begin{equation*}
	\int_{Q_{1/(\ell\delta_\e)}^\nu\cap S_r^\nu}|\nabla u|^2\dx\leq\sum_{z\in J_r^\nu}\int_{Q_{r}^\nu(R_\nu(rz,0))}|\nabla u|^2\dx=\#(J_\e^\nu)\int_{Q_{r}^\nu}|\nabla u|^2\dx\,.
	\end{equation*}
	Thus, since 
	\begin{equation}\label{est:Irnu}
	\#(J_\e^\nu)\leq(\lfloor 1/r_\e\rfloor+1)^{n-1}\leq\Big(\frac{1+r_\e}{r_\e}\Big)^{n-1}\,,
	\end{equation}
	from~\eqref{est:bulk-crit-1} we infer
	\begin{equation}\label{est:bulk-crit-2}
	\int_{Q^\nu}(\bar v_\e^2+\eta_\e)f\Big(\frac{x}{\delta_\e},\nabla \bar u_\e\Big)\dx\leq c_2|\zeta|^2\frac{(1+r_\e)^{n-1}\eta_\e}{\ell\delta_\e}\frac{1}{r^{n-1}}\int_{Q_{r}^\nu}|\nabla u|^2\dx\ \to 0\quad\text{as}\ \e\to 0\,,
	\end{equation}
	where the convergence to zero follows from the fact that $\eta_\e\ll\e$ and $\delta_\e\sim\e$.

	To conclude it only remains to estimate $\F_\e^s(\bar v_\e,Q^\nu)$. Since $\bar v_\e\equiv 1$ on $Q^\nu\setminus S_{r_\e}^\nu$, from~\ref{hyp:growth-h} and a change of variables we deduce that
	\begin{equation}\label{est:Fe-surf-crit1}
	\begin{split}
	\F_\e^s(\bar v_\e,Q^\nu) &=\F_\e^s(\bar v_\e,Q^\nu\cap S_{r_\e}^\nu)=(\ell\delta_\e)^n\int_{Q_{1/(\ell\delta_\e)}^\nu\cap S_r^\nu}\bigg(\frac{(1-v)^2}{\e}+\frac{\e}{(\ell\delta_\e)^2} h(\ell x,\nabla v)\bigg)\dx\\
	&\leq\tilde{\gamma}_\e(\ell\delta_\e)^{n-1}\int_{Q_{1/(\ell\delta_\e)}^\nu\cap S_r^\nu}\big((1-v)^2+h(\ell x,\nabla v)\big)\dx\,,
	\end{split}
	\end{equation}
	where
	\begin{equation}\label{cond:tilde-gammae}
	\tilde{\gamma}_\e\defas\max\left\{\frac{\ell\delta_\e}{\e},\frac{\e}{\ell\delta_\e}\right\}\ \to 1\quad\text{as }\; \e\to 0\,.
	\end{equation}
	Eventually, in view of~\eqref{extension-uv} and~\ref{hyp:per-h} we have
	\begin{equation*}
	\begin{split}
	\int_{Q_{1/(\ell\delta_\e)}^\nu\cap S_r^\nu}\big((1-v)^2+h(\ell x,\nabla v)\big)\dx &\leq\sum_{z\in J_\e^\nu}\int_{Q_{r}^\nu(R_\nu(rz,0))}\big((1-v)^2+h(\ell x,\nabla v)\big)\dx\\
	&=\sum_{z\in J_\e^\nu}\int_{Q_{r}^\nu}\Big((1-v)^2+h\big(\ell x+\ell rR_\nu(z,0),\nabla v\big)\Big)\dx\\
	&=\#(J_\e^\nu)\int_{Q_{r}^\nu}\big((1-v)^2+h(\ell x,\nabla v)\big)\dx\,,
	\end{split}
	\end{equation*}
	where in the last step we also used that $\ell r\in m_\nu\N^*$, hence $\ell rR_\nu(z,0)\in\Z^{n}$ by the choice of $m_\nu$. Thus, using again the estimate on $\#(J_\e^\nu)$ in~\eqref{est:Irnu}, from~\eqref{est:Fe-surf-crit1} we deduce that
	\begin{equation}\label{est:Fe-surf-crit2}
	\F_\e^s(\bar v_\e, Q^\nu)\leq\tilde{\gamma_\e}(1+r_\e)^{n-1}\frac{1}{r^{n-1}}\int_{Q_{r}^\nu}\big((1-v)^2+h(\ell x,\nabla v)\big)\dx
	\end{equation}
	Eventually, by combining~\eqref{almost-optimal-ghomell}, \eqref{est:bulk-crit-2}, \eqref{cond:tilde-gammae}, and \eqref{est:Fe-surf-crit2} we obtain
	\begin{equation*}
	\limsup_{\e\to 0}\F_\e(\bar u_\e,\bar v_\e, Q^\nu)=\limsup_{\e\to 0}\F_\e^s(\bar v_\e,Q^\nu)\leq g_\hom^\ell(\nu)+\eta\,,
	\end{equation*}
	hence~\eqref{c:claim} is proven and thus the claim.
	\end{proof}
		\section{Oscillations on a finer scale than the singular perturbation}\label{sec:finer}
		In this section we characterise $g^\ell$ in the regime $\ell=+\infty$; the latter corresponds to the case where the scale of the oscillations $\delta_\e$ is much smaller than the scale of the singular perturbation $\e$.

\medskip
		
		The following Lemma is a consequence of some analogous results established in~\cite{ABC01} in the context of the homogenisation of Modica-Mortola functionals and will be used in the proof of Proposition \ref{prop:subcritical:case} below. 
	\begin{lemma}\label{lem:ABC}
	Let $\sigma>0$ be fixed and let $A\in\A$. Then there exists $K=K(\sigma)\in\N$ and a constant $c>0$ (depending only on the space dimension) such that for any sequence $(v_\e)\subset W^{1,2}_\loc(\R^n)$ with $\sup_\e \F_\e^s(v_\e,A)<+\infty$ the functions $v_\e^\sigma\in W_\loc^{1,2}(\R^n)$ defined as
	\begin{equation}\label{def:v-eta-ABC}
	v_\e^\sigma(x)\defas\frac{1}{(K\delta_\e)^n}\int_{Q_{K\delta_\e}(x)}v_\e(y)\dy
	\end{equation}
	satisfy the following estimates
	\begin{equation}\label{est:v-eta-ABC}
	\int_A h\bigg(\frac{x}{\delta_\e},\nabla v_\e\bigg)\dx\geq\int_A h_\hom(\nabla v_\e^\sigma)\dx-\sigma\int_{A_{K\delta_\e}}|\nabla v_\e|^2\dx\,,
	\end{equation}
	\begin{equation}\label{est:L2-norm-v-eta}
	\int_A |v_\e-v_\e^\sigma|^2\dx\leq c(K\delta_\e)^2\int_{{A}_{K\delta_\e}}|\nabla v_\e|^2\dx\,,
\end{equation}
	where $h_\hom$ is as in \eqref{h_hom_subcritical} and  ${A}_{K\delta_\e}\defas\Big\{x\in\R^n\colon\dist(x,A)<\frac{(2+\sqrt{n})\sqrt{n}}{2}K\delta_\e\Big\}$. 
\end{lemma}
\begin{proof}
	Estimate \eqref{est:v-eta-ABC} follows by~\cite[Propositions 4.7--4.9]{ABC01}, while estimate~\eqref{est:L2-norm-v-eta} is an $L^2$-version of the $L^1$-estimate obtained in~\cite[Proposition 4.10]{ABC01} and is a direct consequence of Lemma~\ref{lem:est-L2-norm-average} in the appendix.
\end{proof}

	\begin{proposition}\label{prop:subcritical:case}
	Assume that $\ell=+\infty$; moreover, assume that $\eta_\e \simeq \delta_\e \simeq\e^\alpha$,  for some $\alpha >1$. 
	Let $g^\infty_\hom$ be as in~\eqref{g_hom_subcritical} and let $g^\infty$ be the surface integrand in~\eqref{tilde-F}.
	Then for every $(\zeta,\nu)\in\R_0^m\times\S^{n-1}$ there holds
	\begin{equation*}
		{g}^\infty(\zeta,\nu)=g_\hom^\infty(\nu)\,.
	\end{equation*}
	\end{proposition}
	
	\begin{proof}
	Not to overburden notation we still denote by $(\F_\e)$ the $\Gamma$-converging subsequence given by Theorem~\ref{thm:gamma_limit}.
	
We prove the claim in two steps. 
		
		\medskip
		\step 1 $g^\infty(\zeta,\nu)\ge g_{\rm hom}^\infty(\nu)$, for every $(\zeta,\nu)\in\R_0^{m}\times\S^{n-1}$.
		
		\smallskip
		Let $(\zeta,\nu)\in\R_0^m\times\S^{n-1}$ be fixed; by $\Gamma$-convergence we find a sequence $(u_\e,v_\e)\subset W^{1,2}(Q^\nu;\R^m)\times W^{1,2}(Q^\nu)$ converging to $(u_\zeta^\nu,1)$ in $\LtL$ and satisfying
		\begin{equation}\label{eq:recovery-ftilde}
		{g}^\infty(\zeta,\nu)={\F}^\infty(u_\zeta^\nu,1,Q^\nu)=\lim_{\e\to 0}{\F}_\e(u_\e,v_\e,Q^\nu)\,.
		\end{equation}
		By a vectorial truncation argument it is not restrictive to assume that $\|u_\e\|_{L^\infty(Q^\nu;\R^m)}$ is uniformly bounded and such that $u_\e\to u^\nu_\zeta$ also in $L^2(Q^\nu;\R^m)$. Moreover, thanks to Remark~\ref{rem:fund-est} we can apply the fundamental estimate~\cite[Proposition 5.1]{BMZ} to modify $(u_\e,v_\e)$ in such a way that $(u_\e,v_\e)=(\bar{u}_{\zeta,\e}^\nu,\bar{v}_{\e}^\nu)$ near $\partial Q^\nu$ without essentially increasing the energy. Summarising, without loss of generality we can assume to be in the following situation: $(u_\e,v_\e)\to(u_\zeta^\nu,1)$ in $L^2(Q^\nu;\R^m)\times L^2(Q^\nu)$, $(u_\e,v_\e)=(\bar{u}_{\zeta,\e}^\nu,\bar{v}_{\e}^\nu)$ near $\partial Q^\nu$, and $(u_\e,v_\e)$ satisfy~\eqref{eq:recovery-ftilde}. We then extend $(u_\e,v_\e)$ by $(\bar{u}_{\zeta,\e}^\nu,\bar{v}_{\e}^\nu)$ outside $Q^\nu$. \\
		
		Let $0<\sigma<c_1$ be fixed and $v_\e^\sigma$ be given, accordingly, by~\eqref{def:v-eta-ABC} in Lemma~\ref{lem:ABC}. 
	Then the desired inequality can be proven if we show the following: There exist $1<p<\frac43$ (depending on the exponent $\alpha$) and a constant $C>0$ (independent of $\e$ and $\sigma$) such that
		\begin{equation}\label{claim_subcritical}
\liminf_{\e\to 0}\F_\e(u_\e,v_\e,Q^\nu)\ge \liminf_{\e\to 0}\E_\e(u_\e,v_\e^\sigma)-C\sigma\,,
		\end{equation}
		where $\E_\e$ is defined as in~\eqref{AT-anisotropic} with $a=2^{1-p}\sigma$, $b=1-\sigma$, $A=Q^\nu$, and $\varphi=\sqrt{h_\hom}$.
		Indeed, suppose for a moment that~\eqref{claim_subcritical} holds. Then that thanks to~\ref{hyp:hom-h} it is immediate to check that $h_\hom$ is $2$-homogeneous. Since in addition $h_\hom$ is convex and satisfies the growth condition~\ref{hyp:growth-h} (\cf Remark~\ref{rem:prop-g} \eqref{rem:g-infty}), $\sqrt{h_\hom}:\R^n\to[0,+\infty)$ is a norm (see, \eg \cite[Corollary 15.3.1]{Rockafellar}).
		 Then Remark~\ref{rem:AT} together with the fact that $(u_\e,v_\e^\sigma)\to (u_\zeta^\nu,1)$ in $\LtL$ yield 
		\begin{equation*}
\liminf_{\e\to 0}\F_\e(u_\e,v_\e,Q^\nu)\ge{(1-\sigma)}\sqrt{h_\hom(\nu)}-C\sigma\,.
		\end{equation*}
	Thus, gathering~\eqref{eq:recovery-ftilde} and~\eqref{g_hom_subcritical} we obtain 
	\begin{equation*}
		{g}^\infty(\zeta,\nu)\geq {(1-\sigma)}g_\hom^\infty(\nu)-C\sigma\,,
	\end{equation*}
	from which we conclude by letting $\sigma\to 0$.
	
	We are now left to prove~\eqref{claim_subcritical}. To this end we notice that thanks to~Lemma~\ref{lem:ABC} there exists $K=K(\sigma)\in \N$ and $c>0$ such that setting $r=r(\sigma)\defas(2+\sqrt{n})\sqrt{n}K$ we have
		\begin{align}
		\int_{Q^\nu}h\bigg(\frac{x}{\delta_\e},\nabla v_\e\bigg)\dx \geq\int_{Q^\nu}h_{\rm hom}(\nabla v_\e^\sigma)\dx-\sigma\int_{Q_{1+r\delta_\e}^\nu}|\nabla v_\e|^2\dx\,\label{est:h-h-hom}
		\end{align}
		and
		\begin{align}
		\int_{Q^\nu}|v_\e-v_\e^\sigma|^2 \dx \leq c(K\delta_\e)^2\int_{Q_{1+r\delta_\e}^\nu}|\nabla v_\e|^2\dx\,. \label{est:v-v-eta-L2}
		\end{align}
		Moreover, \eqref{eq:recovery-ftilde} together with~\ref{hyp:growth-h} ensure that there exists $M>0$ such that
		\begin{align}\label{est:nabla-vk}
		\int_{Q_{1+r\delta_\e}^\nu}|\nabla v_\e|^2\dx=\int_{Q^\nu}|\nabla v_\e|^2\dx+\int_{Q_{1+r\delta_\e}^\nu\setminus Q^\nu}|\nabla\bar{v}_{\e}^\nu|^2\dx\leq \frac M\e\,,
		\end{align}
		for every $\e$. Therefore, a convexity argument together with~\eqref{est:h-h-hom}--\eqref{est:nabla-vk} yield
		\begin{equation}\label{est:Fk-suf-subcrit}
		\begin{split}
		\F_\e^s(v_\e,Q^\nu) &\geq {(1-\sigma)}\int_{Q^\nu}\frac{(1-v_\e^\sigma)^2}{\e}\dx-\frac{(1-\sigma)}{\sigma}\int_{Q^\nu}\frac{(v_\e-v_\e^\sigma)^2}{\e}\dx+\e\int_{Q^\nu} h\bigg(\frac{x}{\delta_\e},\nabla v_\e\bigg)\dx\\
		&\geq{(1-\sigma)}\int_{Q^\nu}\bigg(\frac{(1-v_\e^\sigma)^2}{\e}+\e h_\hom(\nabla v_\e^\sigma)\bigg)\dx-M\sigma-M \frac{(1-\sigma)}{\sigma}\frac{K^2\delta_\e^2}{\e^2}\,.
		\end{split}
		\end{equation}
	We now turn to estimate the bulk term in  $\F_\e(u_\e,v_\e,Q^\nu)$. 
	Since $\sigma<c_1$, by \ref{hyp:growth-f} and the
	H\"older Inequality with exponents $q=\frac{2}{p}>1$ and $q'=\frac{2}{2-p}>2$ we immediately obtain
	\begin{equation}\label{est-bulk1}
		\int_{Q^\nu}(v_\e^2+\eta_\e)f\bigg(\frac{x}{\delta_\e},\nabla u_\e\bigg)\dx
		\ge
\sigma\int_{Q^\nu}\big(v_\e^p+\eta_\e)|\nabla u_\e|^p\dx-\sigma(1+\eta_\e)\,,
\end{equation}
	for  any $1<p\le2$. 
	Moreover, the convexity inequality $2^{1-p}(a+b)^p\leq a^p+b^p$ gives 
	\begin{equation}\label{est-bulk2}
		\begin{split}
\sigma	\int_{Q^\nu}\big(v_\e^p+\eta_\e)|\nabla u_\e|^p\dx
\ge \frac{\sigma}{2^{p-1}}\int_{Q^\nu}\big((v_\e^\sigma)^p+\eta_\e)|\nabla u_\e|^p\dx- \sigma \int_{Q^\nu}|v_\e^\sigma-v_\e|^p|\nabla u_\e|^p\dx\,.
		\end{split}
	\end{equation}
	Then, thanks to~\eqref{est:Fk-suf-subcrit}--\eqref{est-bulk2} the claim follows if we show that the last term in~\eqref{est-bulk2} vanishes for some suitably chosen $p>1$.
	
	In view of~\eqref{eq:recovery-ftilde} and~\ref{hyp:growth-f} it is not restrictive to assume that
	\begin{equation}\label{est:nabla-uk}
		\eta_\e\int_{Q^\nu}|\nabla u_\e|^2\dx\leq M,
	\end{equation}
	uniformly in $\e$. In this way, again by the H\"older Inequality, we have
	\begin{equation}\label{est:Hoelder2}
		\begin{split}
			\int_{Q^\nu}|v_\e^\sigma-v_\e|^p|\nabla u_\e|^p\dx&\leq\Bigg(\int_{Q^\nu}|v_\e^\sigma-v_\e|^\frac{2p}{2-p}\dx\Bigg)^\frac{2-p}{2}\Bigg(\int_{Q^\nu}|\nabla u_\e|^2\dx\Bigg)^\frac{p}{2}\\
			&\leq M^\frac{p}{2}\eta_\e^{-\frac{p}{2}}\Bigg(\int_{Q^\nu}|v_\e^\sigma-v_\e|^\frac{2p}{2-p}\dx\Bigg)^\frac{2-p}{2}.
		\end{split}
	\end{equation}
		Note that $0\leq v_\e\leq 1$, so that by construction we have $0\leq v_\e^\sigma\leq 1$, hence $|v_\e-v_\e^\sigma|\leq 2$. Since $\frac{2p}{2-p}>2$, this implies that $\big|\frac{v_\e-v_\e^\sigma}{2}\big|^\frac{2p}{2-p}\leq\big|\frac{v_\e-v_\e^\sigma}{2}\big|^2$. Thus, from~\eqref{est:v-v-eta-L2} and~\eqref{est:nabla-vk} we deduce that
	\begin{equation*}
		\int_{Q^\nu}|v_\e-v_\e^\sigma|^\frac{2p}{2-p}\dx\leq 2^{\frac{2p}{2-p}-2}\int_{Q^\nu}|v_\e-v_\e^\sigma|^2\dx\leq 2^\frac{4p-4}{2-p}c(K\delta_\e)^2\e^{-1}M\,.
	\end{equation*}
	Hence, the estimate in~\eqref{est:Hoelder2} gives
	\begin{align}\label{est:Holder-af}
		\int_{Q^\nu}|v_\e^\sigma-v_\e|^p|\nabla u_\e|^p\dx\leq M 2^{2(p-1)}c^\frac{2-p}{2}K^{2-p}\eta_\e^{-\frac{p}{2}}\delta_\e^{2-p}\e^\frac{p-2}{2}\,.
	\end{align}
Since by assumption $\eta_\e \simeq \delta_\e \simeq\e^\alpha$, for some $\alpha >1$, \eqref{est:Holder-af} becomes
\begin{align}\label{est:Hoelder-final}
		\int_{Q^\nu}|v_\e^\sigma-v_\e|^p|\nabla u_\e|^p\dx\leq C\e^{\frac{p-2-3p\alpha}{2}+2\alpha}.	
\end{align}  	 
	We now observe that $\frac{p-2-3p\alpha}{2}+2\alpha>0$, if $p<\frac{4\alpha-2}{3\alpha-1}$; furthermore, the latter can be always fulfilled, since $\frac{4\alpha-2}{3\alpha-1}>1$, for $\alpha>1$. Eventually, with this choice of the exponent $p$ the right hand side of~\eqref{est:Hoelder-final} becomes infinitesimal as $\e \to 0$. Therefore, gathering~ \eqref{est-bulk1},~\eqref{est-bulk2}, and~\eqref{est:Hoelder-final} we get
		\begin{align*}
				 \int_{Q^\nu}(v_\e^2+\eta_\e)f\bigg(\frac{x}{\delta_\e},\nabla u_\e\bigg)\dx\ge
		\frac{\sigma}{2^{p-1}}\int_{Q^\nu}\big((v_\e^\sigma)^p+\eta_\e)|\nabla u_\e|^p\dx
		-\sigma+o(1)\,,
	\end{align*}
	as $\e\to 0$. 
		Together with~\eqref{est:Fk-suf-subcrit} this gives \eqref{claim_subcritical} with $C=M+1$ and thus the desired inequality.
		
		\medskip
		\step 2 $g^\infty(\zeta,\nu)\le g_{\rm hom}^\infty(\nu)$, for every $(\zeta,\nu)\in\R_0^m\times\S^{n-1}$.
		
	\smallskip
	Let $(\zeta,\nu)\in\R_0^m\times\S^{n-1}$ be fixed; by Theorem~\ref{thm:gamma_limit} we have that 
	\begin{equation}\label{est:g-tilde-liminf}
		g^\infty(\zeta,\nu)=\F^\infty(u_\zeta^\nu,1,Q^\nu)\le\liminf_{\e\to 0}\F_\e(u_\e,v_\e,Q^\nu)\,,
	\end{equation}
	for every sequence $(u_\e,v_\e)$ converging to $(u^{\nu}_\zeta,1)$ in $L^0(\R^n;\R^m)\times L^0(\R^n)$. 
	
	Let $\eta>0$ be fixed, in what follows we construct a sequence $(\bar u_\e,\bar v_\e)$ converging to $(u_\zeta^{\nu},1)$ in $L^0(\R^n;\R^m)\times L^0(\R^n)$ and such that 
	\begin{equation}\label{claim2}
		\limsup_{\e\to 0}\F_\e(\bar u_\e,\bar v_\e,Q^\nu)\le g^\infty_{\rm hom}(\nu)+\eta\,.
	\end{equation}
	Then, we can conclude by combining~\eqref{est:g-tilde-liminf} and~\eqref{claim2} and by the arbitrariness of $\eta>0$.
	
	We observe that since $g^\infty(\zeta,\cdot)$ and $g_{\rm hom}^\infty$ are both continuous (\cf Remarks~\ref{rem:prop-g}~\eqref{rem:g-infty} and~\ref{rem:prop-g-tilde}), by a standard density argument it is enough to consider $\nu\in\S^{n-1}\cap \Q^\nu$. Let then $R_\nu\in\Q^{n\times n}$ be an orthogonal matrix as in~\ref{Rn} such that $R_\nu e_n=\nu$.
	Then there is $m_\nu \in \N^*$ such that $m_\nu R_\nu\in\Z^{n\times n}$. In this way, we have
	\begin{equation}\label{cond:m-nu}
	m_\nu R_\nu(z,0)\in\Pi^\nu\cap\Z^n\; \text{ for all }\, z\in\Z^{n-1}\quad\text{and}\quad m_\nu\nu\in\Z^n\,.
	\end{equation}
We now define the sequence $(\bar{u}_\e,\bar{v}_\e)$. To this end, let $\xi_\e\defas\lfloor\frac{\sqrt{\e\eta_\e}}{\delta_\e}\rfloor\delta_\e$, and define $\bar{u}_\e\in W^{1,2}_\loc(\R^n;\R^m)$ by setting
	\begin{equation*}
 	\bar	u_\e(x)\defas
 		\begin{cases}
 			\dfrac{x\cdot{\nu}}{\xi_\e}u_\zeta^{\nu}&\text{if}\ |x\cdot\nu|\le \xi_\e\,,\\
 			u_\zeta^{\nu}&\text{otherwise}\,,
 		\end{cases}
 	\end{equation*}
 	so that in particular $\nabla\bar{u}_\e=0$ outside $\{|x\cdot\nu|\le\xi_\e\}$. We now define $\bar{v}_\e \in W^{1,2}_\loc(\R^n)$ in such a way that $\bar{v}_\e=0$ in the region where $\nabla\bar{u}_\e\neq 0$. 
To this end, note that by Remark~\ref{rem:op} for a fixed $\eta>0$ there exist $T_\eta>0$ and $v_\eta\in W^{1,2}(0,T_\eta)$ with $v_\eta(0)=0$, $v_\eta(T_\eta)=1$ such that
	\begin{equation}\label{eq:optimal_profile}
	\int_0^{T_\eta}\big((1-v_\eta)^2+h_{\rm hom}(\nu)(v'_\eta)^2\big)\dt\le\sqrt{h_{\rm hom}(\nu)}+\frac{\eta}{2}\,,
	\end{equation}
	where $h_{\rm hom}$ is defined in \eqref{h_hom_subcritical}. Let us extend $v_\eta$ to $(0,+\infty)$ by setting $v_\eta(t)\defas 1$ for $t\geq T_\eta$ and let us define $v_\eta^\nu\in W^{1,2}_\loc(\R^n)$ by setting $v_\eta^\nu(x)\defas v_\eta(|x\cdot\nu|)$.

	By appealing to the classical homogenization result (see, \eg \cite[Proposition 11.7 and Theorem 14.5]{BD}), for any positive sequence $\sigma\to0$ we deduce the existence of a sequence $(w_\sigma^+)\subset W^{1,2}\big(R_\nu(Q'\times(0,T_\eta))\big)$ such that $(w_\sigma^+-v_\eta^\nu)\in W_0^{1,2}\big(R_\nu(Q'\times(0,T_\eta))\big)$,
		$w_\sigma^+\to v_\eta^\nu$ in $L^2(R_\nu(Q'\times(0,T_\eta)))$ as $\sigma\to0$ and
	\begin{equation}\label{eq:upbpund_sub_0}
	\begin{split}
	&\lim_{\sigma\to0}\int_{R_\nu(Q'\times(0,T_\eta))}\bigg((1-w_\sigma^+)^2+ h\left(\frac{x}{\sigma},\nabla w_\sigma^+\right)\bigg)\dx\\
	&\quad=\int_{R_\nu(Q'\times(0,T_\eta))}\big((1-v_\eta^\nu)^2+h_{\rm hom}(\nabla v_\eta^\nu)\big)\dx\leq\sqrt{h_\hom(\nu)}+\frac{\eta}{2}\,,
	\end{split}
	\end{equation}
	where the last inequality follows from~\eqref{eq:optimal_profile} and the definition of $v_\eta^\nu$.
	Similarly, we find $(w_\sigma^-)\subset W^{1,2}\big(R_\nu(Q'\times(-T_\eta,0))\big)$ satisfying the analogous properties on $R_\nu(Q'\times(-T_\eta,0))$.
	
	Now set $\sigma_\e\defas\frac{\delta_\e}{\e}\to 0$, $w_\e^+\defas w_{\sigma_\e}^+$, and $w_\e^-\defas w_{\sigma_\e}^-$. By a  truncation argument it is not restrictive to assume that $0\leq w_\e^-,w_\e^+\leq 1$.
		Moreover, let 
	\begin{equation}\label{def:re}
	r_\e\defas\Big(\Big\lfloor\frac{\e}{m_\nu\delta_\e}\Big\rfloor+1\Big)m_\nu\delta_\e\in (\e,\e+m_\nu\delta_\e)
	\end{equation}
	 and consider the open intervals 
	 \begin{equation*}
	 I_\e^-\defas\big(-m_\nu\xi_\e-\e T_\eta\,,-m_\nu\xi_\e\big)\quad\text{and}\quad I_\e^+\defas \big(m_\nu\xi_\e,m_\nu\xi_\e+\e T_\eta\big)\,,
	 \end{equation*}
	 as well as $I_\e\defas I_\e^-\cup I_\e^+$. We start defining $\bar v_\e$ on $R_\nu\big(Q'_{r_\e}\times I_\e\big)$ by setting
	 \begin{equation*}
	   \bar	v_\e(x)\defas\begin{cases}
	   w_{\e}^-\left(\dfrac{x+m_\nu\xi_\e\nu}{\e}\right) &\text{in}\ R_\nu(Q_\e'\times I_\e^-)\,,\\[1em]
	   w_{\e}^+\left(\dfrac{x-m_\nu\xi_\e\nu}{\e}\right) &\text{in}\ R_\nu(Q_\e'\times I_\e^+)\,,\\[1em]
	 	v_\eta\left(\dfrac{|x\cdot\nu|-m_\nu\xi_\e}{\e}\right)&\text{in}\ R_\nu\big((Q_{r_\e}'\setminus Q_\e')\times I_\e\big)\,.\\
	 		\end{cases}
	 \end{equation*}
	We then extend $\bar v_\e$ $r_\e$-periodically in directions $R_\nu e_i$, $i=1,\ldots, n-1$ by setting
	\begin{equation}\label{per-extension-ve}
	\bar{v}_\e(x)\defas\bar{v}_\e\big(x-R_\nu(r_\e z,0)\big)\; \text{ if }\; x\in R_\nu\big((Q_{r_\e}'+r_\e z)\times I_\e\big)\,,\ z\in\Z^{n-1}\,. 
	\end{equation}
	In this way, $\bar{v}_\e$ is defined on the set 
	\[
	R_\nu(\R^{n-1}\times I_\e)=\{x\in \R^n \colon m_\nu\xi_\e<|x\cdot\nu|<m_\nu\xi_\e+\e T_\eta\}
	\] 
	and eventually we extend $\bar{v}_\e$ by setting $\bar{v}_\e(x)\defas 1$ if $|x\cdot\nu|\geq \e T_\eta+m_\nu\xi_\e$ and $\bar{v}_\e(x)\defas 0$ if $|x\cdot\nu|\leq m_\nu\xi_\e$. Note that thanks to the boundary conditions satisfied by $w_\e^+$ and $w_\e^-$ the functions $\bar{v}_\e$ belong to $W_\loc^{1,2}(\R^n)$. Moreover, by construction we have
	\begin{equation*}
		(\bar u_\e(x),\bar v_\e(x))=(u_\zeta^\nu(x),1)\quad\text{ if  }\ |x\cdot\nu|\ge m_\nu\xi_\e+\e T_\eta\,,
	\end{equation*}
	and thus $(\bar u_\e,\bar v_\e)$ converges to $(u_\zeta^\nu,1)$ in 
	$L^0(\R^n;\R^m)\times L^0(\R^n)$. 
	
To conclude it is then left to show that $(\bar{u}_\e,\bar{v}_\e)$ also satisfies~\eqref{claim2}. Since $\bar v_\e \nabla \bar u_\e =0$ a.e.\! in $\R^n$, from~\ref{hyp:growth-f} we deduce that
	\begin{equation}\label{eq:upbound1}
	{\F}_\e(\bar u_\e,\bar v_\e,Q^\nu)\leq c_2\eta_\e\int_{Q^\nu}|\nabla\bar u_\e|^2\dx+	\F^{s}_\e(\bar v_\e,Q^{\nu})\leq c_2|\zeta|^2\frac{\eta_\e}{\xi_\e}+\F_\e^s(\bar{v}_\e, Q^\nu)\,,
	\end{equation}
where to establish \eqref{eq:upbound1} we also used that $|\nabla \bar{u}_\e|\leq\frac{|\zeta|}{\xi_\e}$ in $\{|x\cdot\nu|\leq\xi_\e\}$ and $\nabla\bar{u}_\e=0$ outside. By the choice of $\xi_\e$ the first term on the right-hand side of~\eqref{eq:upbound1} vanishes as $\e\to 0$; hence it only remains to estimate $\F_\e^s(\bar{v}_\e,Q^\nu)$. Setting $S_\e\defas R_\nu(Q'\times I_\e)$ we have
		\begin{equation}\label{eq:upbound_sub_1}
			\begin{split}
				\F^{s}_\e(\bar v_\e,Q^{\nu})\le\frac{2 m_\nu\xi_\e}{\e}+ \int_{S_\e}\bigg(\frac{(1-\bar v_\e)^2}{\e}+\e h\Big(\frac{x}{\delta_\e},\nabla \bar v_\e\Big)\bigg)\dx\,.
			\end{split}
		\end{equation}
	To estimate the second term on the right-hand side of \eqref{eq:upbound_sub_1} define 
\begin{equation*}
	J_\e\defas\Big\{z\in\Z^{n-1}\colon (Q'_{r_\e}+ r_\e z)	\cap Q'\ne\emptyset	\Big\}\,;
\end{equation*}
	then, in view of~\eqref{per-extension-ve}, the periodicity assumption~\ref{hyp:per-h} together with~\eqref{cond:m-nu} and the choice of $r_\e$ imply that
	\begin{equation}\label{est:ub-subcrit-periodicity}
	\begin{split}
	\int_{S_\e}\bigg(\frac{(1-\bar v_\e)^2}{\e}+\e h\Big(\frac{x}{\delta_\e},\nabla \bar v_\e\Big)\bigg)\dx &\leq\sum_{z\in J_\e}\int_{R_\nu\big((Q_{r_\e}'+r_\e z)\times I_\e\big)}\bigg(\frac{(1-\bar v_\e)^2}{\e}+\e h\Big(\frac{x}{\delta_\e},\nabla \bar v_\e\Big)\bigg)\dx\\
	&=\#(J_\e)\int_{R_\nu(Q_{r_\e}'\times I_\e)}\bigg(\frac{(1-\bar v_\e)^2}{\e}+\e h\Big(\frac{x}{\delta_\e},\nabla \bar v_\e\Big)\bigg)\dx\,.
	\end{split}
	\end{equation}
	We now estime the integral on the right-hand side of~\eqref{est:ub-subcrit-periodicity} on the set $R_\nu(Q_{\e}'\times I_\e^+)$ by applying the change of variables $y=\e x+m_\nu\xi_\e\nu$. Since $\xi_\e\in\delta_\e\Z$, recalling~\eqref{cond:m-nu} and~\ref{hyp:per-h} we obtain
		\begin{equation}\label{eq:upbound_sub_2}
		\begin{split}
		\int_{R_\nu(Q_\e'\times I_\e^+)}\bigg(\frac{(1-\bar v_\e)^2}{\e}+\e h\Big(\frac{x}{\delta_\e},\nabla\bar v_\e\Big)\bigg)\dx=\e^{n-1}\int_{R_\nu(Q'\times(0,T_\eta))}\bigg((1-w_{\e}^+)^2+h\Big(\frac{x}{\sigma_\e},\nabla w_{\e}^+\Big)\bigg)\dx\,.
		\end{split}
		\end{equation}
		Moreover, a similar estimate holds on $R_\nu(Q_\e'\times I_\e^-)$. Instead, on the remaining part of $R_\nu(Q_{r_\e}'\times I_\e)$ the definition of $\bar{v}_\e$ together with a change of variables, Fubini's Theorem, and~\ref{hyp:growth-h} give
	\begin{equation}\label{eq:upbound_sub_3}
	\begin{split}
	&\int_{R_\nu((Q_{r_\e}'\setminus Q_\e')\times I_\e)}\bigg(\frac{(1-\bar v_\e)^2}{\e}+\e h\Big(\frac{x}{\delta_\e},\nabla \bar v_\e\Big)\bigg)\dx\leq\int_{R_\nu((Q_{r_\e}'\setminus Q_\e')\times I_\e)}\bigg(\frac{(1-\bar v_\e)^2}{\e}+c_4\e|\nabla\bar{v}_\e|^2\bigg)\dx\\
	&\quad=2\int_{Q_{r_\e}'\setminus Q_\e'}\int_0^{T_\eta}\big((1-v_\eta)^2+c_4(v'_\eta)^2\big)\dt\dx'\leq 2\big(r_\e^{n-1}-\e^{n-1}\big)\frac{c_4}{c_3}C_\eta\,.
	\end{split}
	\end{equation}
We finally observe that $\#(J_\e)\le (\lfloor 1/r_\e\rfloor+1)^{n-1}$ and that thanks to~\eqref{def:re} we have $\frac{r_\e}{\e}=1+O(\frac{\delta_\e}{\e})$.
Thus, gathering~\eqref{eq:upbound_sub_1}--\eqref{eq:upbound_sub_3} we deduce that
		\begin{equation}\label{eq:upbound_sub_4}
			\begin{split}
				\F_\e^s(\bar{v}_\e,Q^\nu) 
				&
				\le (1+r_\e)^{n-1}\int_{R_\nu(Q'\times(0,T_\eta))}\bigg((1-w_\e^+)^2+h\left(\frac{x}{\sigma_\e},\nabla w_\e^+\right)\bigg)\dx\\
				&+ (1+r_\e)^{n-1}\int_{R_\nu(Q'\times(-T_\eta,0))}\bigg((1-w_\e^-)^2+h\left(\frac{x}{\sigma_\e},\nabla w_\e^-\right)\bigg)\dx+o(1)\,,
			\end{split}
		\end{equation}
as $\e \to 0$.		
		Eventually, by combining~\eqref{eq:upbound1}, \eqref{eq:upbound_sub_4}, and \eqref{eq:upbpund_sub_0} we get
	\begin{equation*}
			\limsup_{\e\to0}\F_\e(\bar u_\e,\bar v_\e,Q^\nu)\le 2\sqrt{h_{\rm hom}(\nu)}+\eta\,,
		\end{equation*}
		hence~\eqref{claim2} is proven and thus the claim.
	\end{proof}
		\appendix
	\section*{Appendix}
	\setcounter{theorem}{0}
	\setcounter{equation}{0}
	\renewcommand{\theequation}{A.\arabic{equation}}
	\renewcommand{\thetheorem}{A.\arabic{theorem}}
	In this short section we state and prove Lemma~\ref{lem:est-L2-norm-average}, from which estimate~\eqref{est:L2-norm-v-eta} in Lemma~\ref{lem:ABC} follows. A similar estimate has been obtained in the proof of~\cite[Proposition 4.10]{ABC01} for the $L^1$-norm. 
	
	\begin{lemma}\label{lem:est-L2-norm-average} There exists $c=c(n)>0$ such that for every open sets $A, A'\subset\R^n$, with $A'\wcont A$, every $v\in W^{1,2}_\loc(\R^n)$, and every $r$ satisfying  
	\[ 
	0<r<\frac{2\dist(A',\partial A)}{(2+\sqrt{n})\sqrt{n}},
	\] 
	there holds 
	\begin{equation*}\label{est:v-vr-L2}
	\int_{A'}|v-v_r|^2\dx\leq c\,r^2\int_A|\nabla v|^2\dx\,,
	\end{equation*}
	where
	\begin{equation*}\label{def:vr}
	v_r(x)\defas\frac{1}{r^n}\int_{Q_r(x)}v(y)\dy\,.
	\end{equation*}
	\end{lemma}
	\begin{proof}
	Let $A,A' \subset \R^n$ be open and such that $A'\wcont A$; for any $r>0$ set
	\begin{equation*}
	J_r\defas\{z\in\Z^n\colon Q_r(rz)\cap A'\neq \emptyset\}\,.
	\end{equation*}
	For any $v\in W^{1,2}_\loc(\R^n)$ we have
	\begin{equation}\label{est:append-1}
	\begin{split}
	\int_{A'}|v(x)-v_r(x)|^2\dx&\leq\sum_{z\in J_r}\int_{Q_r(rz)}|v(x)-v_r(x)|^2\dx\\
	&\leq 2\sum_{z\in J_r}\bigg(\int_{Q_r(rz)}|v(x)-v_r(rz)|^2+\int_{Q_r(rz)}|v_r(rz)-v_r(x)|^2\dx\bigg)\,.
	\end{split}
	\end{equation}
	We estimate the two terms on the right-hand side of~\eqref{est:append-1} using the Poincar\'{e} Inequality and the continuity of the translation operator. 
	
	By a scaling argument, for every $z\in J_r$ we have 
	\begin{equation}\label{est:poincare}
	\|v-v_r(rz)\|_{L^2(Q_r(rz))}\leq c_{P} r\|\nabla v\|_{L^2(Q_r(rz);\R^n)}\,,
	\end{equation}
	where $c_P>0$ is the constant for the Poincar\'{e} inequality in the unit cube. 
	Also note that for any $z\in J_r$ and $x\in Q_r(rz)$ there exists $x_0\in A'\cap Q_r(rz)$ with $|x-x_0|<\sqrt{n}r$, hence
	\begin{equation}\label{inclusion1}
	Q_r(rz)\subset A\quad\text{as long as}\ r<\frac{\dist(A',\partial A)}{\sqrt{n}}\,.
	\end{equation}
	Moreover, the cubes $Q_r(rz)$, $z\in J_r$ are pairwise disjoint.
	Thus, summing up~\eqref{est:poincare}, for $r>0$ satisfying~\eqref{inclusion1} we obtain
	\begin{equation}\label{est:append-2}
	\sum_{z\in J_r}\int_{Q_r(rz)}|v(x)-v_r(rz)|^2\leq c_P^2r^2\int_{A}|\nabla v|^2\dx.
	\end{equation}
	Then, it remains to estimate the second term in~\eqref{est:append-1}. An application of Jensen's Inequality yields for any $z\in J_r$
	\begin{equation}\label{est:jensen}
	\begin{split}
	\int_{Q_r(rz)}|v_r(rz)-v_r(x)|^2\dx&=\int_{Q_r(rz)}\bigg|\frac{1}{r^n}\int_{Q_r(rz)}v(y)-v(y+x-rz))\dy\bigg|^2\dx\\
	&\leq\frac{1}{r^n}\int_{Q_r(rz)}\int_{Q_r(rz)}|v(y)-v(y+x-rz)|^2\dy\dx\,.
	\end{split}
	\end{equation}
	Since $|x-rz|\leq\sqrt{n}r/2$ for any $x\in Q_r(rz)$, the continuity of the shift operator in Sobolev spaces (see, \eg \cite[Proposition 9.3]{brezis}) yields
	\begin{equation*}
	\big\|v-v\big(\cdot+(x-rz)\big)\big\|_{L^2(Q_r(rz))}\leq\frac{\sqrt{n}}{2}r\|\nabla v\|_{L^2(Q_{(1+\sqrt{n})r}(rz);\R^n)},\quad\text{for every}\ x\in Q_r(rz)\,
	\end{equation*}
	and therefore
	\begin{equation}\label{est:append-3}
	\begin{split}
	\frac{1}{r^n}\int_{Q_r(rz)}\int_{Q_r(rz)}|v(y)-v(y+x-rz)|^2\dy\dx &\leq\frac{n}{4}r^2\int_{Q_{(1+\sqrt{n})r}(rz)}|\nabla v(y)|^2\dy\,.
	\end{split}
	\end{equation}
	Moreover, for any $z\in J_r$ and $x\in Q_{(1+\sqrt{n})r}(rz)$ there exists $x_0\in Q_r(rz)\cap A'$ with $|x-x_0|<(2+\sqrt{n})\sqrt{n}r/2$, hence
	\begin{equation}\label{inclusion2}
	Q_{(1+\sqrt{n})r}(rz)\subset A\quad\text{if }\ r<\frac{2\dist(A',\partial A)}{(2+\sqrt{n})\sqrt{n}}\,.
	\end{equation}
	We observe that the cubes $Q_{(1+\sqrt{n})r}(rz)$, $z\in J_r$ are not pairwise disjoint. Nevertheless, since for any $z_1, z_2\in J_r$ with $Q_{(1+\sqrt{n})r}(rz_1)\cap Q_{(1+\sqrt{n})r}(rz_2)\neq\emptyset$ we have $|z_1-z_2|\leq (1+\sqrt{n})\sqrt{n}$, each cube $Q_{(1+\sqrt{n})r}(rz)$ intersects only $N$ cubes, with $N$ independent of $r$. Thus, summing up the estimates in~\eqref{est:jensen} and~\eqref{est:append-3}, for $r>0$ as in~\eqref{inclusion2} we obtain
	\begin{equation}\label{est:append-4}
	\sum_{z\in J_r}\int_{Q_r(rz)}|v_r(rz)-v_r(x)|^2\dx\leq N\frac{n}{4}r^2\int_A|\nabla v|^2\dx\,.
	\end{equation}
	Eventually the claim follows by combining~\eqref{est:append-1}, \eqref{est:append-2}, and~\eqref{est:append-4}.
	\end{proof}

\section*{Acknowledgments}
The work of R. Marziani and C. I. Zeppieri was supported by the Deutsche Forschungsgemeinschaft (DFG, German Research Foundation) project 3160408400 and under the Germany Excellence Strategy EXC 2044-390685587, Mathematics M\"unster: Dynamics--Geometry--Structure. The work of A. Bach was supported by the National Research Project PRIN 2017 CUP B88D19000330001.

\end{document}